\documentclass[11pt,reqno]{amsart}

\usepackage[a4paper,left=35mm,right=35mm,top=30mm,bottom=30mm,marginpar=25mm]{geometry}

\usepackage{amsmath}
\usepackage{amssymb}
\usepackage{amsthm}
\usepackage[utf8]{inputenc}
\usepackage{graphicx}
\usepackage{dsfont}
\usepackage[nocompress, space]{cite}

\allowdisplaybreaks

\usepackage{algorithm}
\usepackage{algorithmic}
\usepackage{enumitem}
\usepackage{hyperref}

\newcommand{\Rr}{{\mathbb{R}}}

\newcommand{\Tt}{{\mathbb{T}}}

\newcommand{\epsi}{\varepsilon}

\def\geq{\geqslant}

\numberwithin{equation}{section}

\newtheoremstyle{thmlemcorr}{10pt}{10pt}{\itshape}{}{\bfseries}{.}{10pt}{{\thmname{#1}\thmnumber{
#2}\thmnote{ (#3)}}}
\newtheoremstyle{thmlemcorr*}{10pt}{10pt}{\itshape}{}{\bfseries}{.}\newline{{\thmname{#1}\thmnumber{
#2}\thmnote{ (#3)}}}
\newtheoremstyle{defi}{10pt}{10pt}{\itshape}{}{\bfseries}{.}{10pt}{{\thmname{#1}\thmnumber{
#2}\thmnote{ (#3)}}}
\newtheoremstyle{remexample}{10pt}{10pt}{}{}{\bfseries}{.}{10pt}{{\thmname{#1}\thmnumber{
#2}\thmnote{ (#3)}}}
\newtheoremstyle{ass}{10pt}{10pt}{}{}{\bfseries}{.}{10pt}{{\thmname{#1}\thmnumber{
A#2}\thmnote{ (#3)}}}

\theoremstyle{thmlemcorr}
\newtheorem{theorem}{Theorem}
\numberwithin{theorem}{section}
\newtheorem{lemma}[theorem]{Lemma}

\newtheorem{proposition}[theorem]{Proposition}

\theoremstyle{thmlemcorr*}
\newtheorem{theorem*}{Theorem}
\newtheorem{lemma*}[theorem]{Lemma}
\newtheorem{corollary*}[theorem]{Corollary}
\newtheorem{proposition*}[theorem]{Proposition}
\newtheorem{problem*}[theorem]{Problem}
\newtheorem{conjecture*}[theorem]{Conjecture}

\theoremstyle{defi}
\newtheorem{definition}[theorem]{Definition}
\newtheorem{assumption}[theorem]{Assumption}

\newtheorem{problem}{Problem}

\theoremstyle{remexample}
\newtheorem{remark}[theorem]{Remark}
\newtheorem{example}[theorem]{Example}

\newtheorem{teo}[theorem]{Theorem}

\theoremstyle{ass}

\newcommand{\T}{{\mathbb{T}^{d}}}
\newcommand{\pair}[2]{\left\langle #1,#2\right\rangle}
\newcommand{\diver}{\operatorname{div}}

\newcommand{\VI}{\operatorname{VI}}
\newcommand{\Mirr}{\operatorname{Mirr}}
\newcommand{\eps}{\epsi}
\newcommand{\RR}{\mathbb{R}}
\newcommand{\bbar}{\bar\beta}
\newcommand{\gbar}{\bar\gamma}

\title[Bregman-projected mirror methods for MFG]{Bregman-projected mirror methods for regularized stationary mean-field games}

\author{Hussain Al Abdulaziz}
\address{King Abdullah University of Science and Technology (KAUST), CEMSE Division, Thuwal 23955-6900, Saudi Arabia}
\email{hussain.abdulaziz.1@kaust.edu.sa}

\author{Yuri Ashrafyan}
\address{King Abdullah University of Science and Technology (KAUST), CEMSE Division, Thuwal 23955-6900, Saudi Arabia}
\email{yuri.ashrafyan@kaust.edu.sa}

\author{Yeva Gevorgyan}
\address{King Abdullah University of Science and Technology (KAUST), CEMSE Division, Thuwal 23955-6900, Saudi Arabia}
\email{yeva.gevorgyan@kaust.edu.sa}

\author{Diogo Gomes}
\address{King Abdullah University of Science and Technology (KAUST), CEMSE Division, Thuwal 23955-6900, Saudi Arabia}
\email{diogo.gomes@kaust.edu.sa}

\date{October 16, 2026}



\thanks{D.\ Gomes was partially supported by KAUST baseline funds and KAUST SRI, Center for Uncertainty Quantification in Computational Science and Engineering.}

\begin{document}

\begin{abstract}
We develop and analyze a Bregman-projected mirror iteration for
low-order regularizations of stationary mean-field game (MFG) systems in their
natural Banach space setting.
For separable Hamiltonians of the form
$H(x,p,m)=H_0(x,p)-g(m)$, with quadratic or super-quadratic Hamiltonian
growth and linear or super-linear density couplings, we formulate 
a
low-order $\bar\gamma$-Laplacian regularization of the stationary MFG system
as a variational inequality on
$L^{\bar\beta}(\mathbb T^d)\times W^{1,\bar\gamma}(\mathbb T^d)$.
To approximate solutions of this regularized variational inequality, we
introduce a Bregman geometry matched to the mixed Lebesgue--Sobolev exponents
of the problem and analyze the exact constrained one-step mirror iteration and
a frozen two-step variant. Under quantitative monotonicity of the coupling,
both exact iterations converge strongly to the unique solution of the
regularized variational inequality for each fixed $\varepsilon>0$, under natural
summability conditions on the step sizes. Along weakly convergent
vanishing-regularization sequences, we also prove strong convergence in the
full product space. Numerical experiments in one and two dimensions show
first-order accuracy on manufactured solutions and weak dependence of outer
iteration counts on the tested grids. At matched effective step sizes, the
frozen two-step and one-step implementations agree up to inner-solve
accuracy, while the latter requires fewer mirror solves.
\end{abstract}

\maketitle

\section{Introduction}
\label{sec:intro}

Mean-field game (MFG) theory emerged from two complementary lines of work.
Lasry and Lions introduced the PDE approach in the stationary and finite-horizon
settings in \cite{ll1,ll2}, later synthesized in
\cite{lasryMeanFieldGames2007}. Independently, Huang, Malham\'e, and Caines
developed the Nash certainty equivalence approach to large-population stochastic
dynamic games in \cite{huangLargePopulationStochastic2006}, with the LQG
\(\epsi\)-Nash theory developed further in \cite{Caines2}. MFGs model
strategic interactions in large populations of rational agents and lead, in the
stationary case, to coupled nonlinear PDE systems consisting of a
Hamilton--Jacobi equation for the value function \(u\) and a transport equation
for the population density \(m\).

In this paper, we study numerical methods for solving stationary MFGs with a
separable Hamiltonian structure. In the power-growth regime considered below,
\(\alpha\ge2\) denotes the Hamiltonian growth exponent and \(\beta\ge1\) the
coupling exponent; we set
\[
\bar\beta:=\beta+1,
\qquad
\bar\gamma:=\alpha\frac{\beta+1}{\beta}.
\]
The unregularized stationary problem motivating the analysis is formulated as follows.
\begin{problem}\label{prob:main}
Let \(X = L^{\bar{\beta}}(\mathbb{T}^d) \times W^{1,\bar{\gamma}}(\mathbb{T}^d)\).
Given a convex Hamiltonian \(H_0\), a density coupling \(g\), and a potential
\(V\), find \((m,u)\in X\), with \(m\ge0\) a.e., solving the variational
inequality associated with the formal system
\begin{equation}\label{eq:main}
\begin{cases}
u + H_0(x,Du) \le g(m)+V(x),\\
u + H_0(x,Du) = g(m)+V(x) \quad \text{a.e. on } \{m>0\},\\
-\diver(m D_pH_0(x,Du)) + m = 1.
\end{cases}
\end{equation}
\end{problem}

The variational inequality formulation is made precise after introducing the
operator \(\mathcal A_0\) in Section~\ref{sec:reg_operator_properties}.
The coupling between the two equations in \eqref{eq:main} is the main source of
the system's numerical difficulty and motivates the iterative methods studied here.
Problems of this form have been studied extensively. Existence results for stationary and time-dependent MFGs based on monotone operator theory were developed in \cite{FG2,FGT1,FeGoTa21} in Hilbert spaces using high-order regularizations. Recently, a Banach-space framework using low-order regularization was established in \cite{ferreiraSolvingMeanFieldGames2025}, obtaining strong solutions for power-growth models. On the numerical side, variational and monotone methods were explored in an $L^2$ setting in \cite{almullaTwoNumericalApproaches2017}, and preconditioned extragradient methods for time-dependent cases in \cite{GG26}. 
Mirror descent and Bregman projections have also recently appeared in the MFG literature, but in settings whose geometry differs from the present one. In \cite{PPELP21,wuPopulationawareOnlineMirror2024}, online mirror descent with the Kullback--Leibler divergence is applied to finite-state, discrete-time MFGs on finite-dimensional probability simplices. 
In \cite{nurbekyanMonotoneInclusionMethods2024}, time-dependent MFGs are first discretized and the resulting finite-difference system is solved by the monotone-inclusion version of primal-dual hybrid gradient splitting analyzed in \cite{nurbekyanNoteConvergenceMonotone2023}.
 Mirror descent has also been applied to mean-field \emph{control} (rather than games) in \cite{LaurMD23}, and a primal-dual partial inverse splitting for MFGs with nonlocal couplings is developed in \cite{BADLS23}.
The mirror-descent and Bregman-projection methodology itself originates in \cite{NY83} and has been developed in \cite{BT03,AKL22,Sem17}, with applications to variational inequalities with monotone operators in \cite{HR22,IRS23}.

Despite this progress, 
we are not aware of a convergence theory for constrained Bregman mirror 
iterations posed directly in 
\(L^{\bar\beta}(\mathbb T^d)\times W^{1,\bar\gamma}(\mathbb T^d)\)
for stationary MFG variational inequalities.

A convergence analysis in the natural product norm must account for
the duality between \(X\) and \(X^*\) and the different integrability
exponents of the density and value function.
Mirror descent uses the geometry induced by a strictly convex
functional. This geometry may be adapted to the mixed regularity of
\(m \in L^{\bar{\beta}}\) and \(u \in W^{1,\bar{\gamma}}\).
Resolving this gap is significant because $L^{\bar{\beta}} \times W^{1,\bar{\gamma}}$ is the natural space for the variational formulation \cite{ferreiraSolvingMeanFieldGames2025}, and an algorithm whose convergence is established in that norm is therefore directly compatible with the existence and a priori theory.

Here, we prove strong convergence of a Bregman-projected mirror scheme for a
regularized Banach-space variational inequality associated with the stationary
MFG. Our precise hypotheses on the separable Hamiltonian
\(H(x,p,m)=H_0(x,p)-g(m)\) and the power-type density coupling are stated in
Section~\ref{sec:assumptions}; a representative example covered by our
hypotheses is \(H_0(x,p)=a(x)|p|^\alpha\) and
\(g(m)=c m^{\bar\beta-1}\). The variational setting of
\cite{ferreiraSolvingMeanFieldGames2025} provides the foundation for our
work.
The paper makes three contributions.
First, we construct a mixed-exponent Bregman geometry adapted to the
stationary MFG variational inequality and prove strong convergence of the
exact constrained one-step mirror iteration and of a frozen two-step
variant for each fixed regularization parameter. The proof uses quantitative
monotonicity of the coupling and boundedness of the regularized operator on
the admissible set, without imposing Lipschitz continuity.
Second, quantitative monotonicity also yields uniqueness of the regularized
solution and strong convergence in the full product space along weakly
convergent vanishing-regularization sequences.
Third, the numerical experiments validate the discretization against
manufactured solutions and compare the implementations at matched effective
step sizes, identifying when the second frozen mirror solve is redundant.
 
Our main result is the following.
\begin{teo}\label{thm:main}
Let Assumptions~\ref{ass:base} and~\ref{ass:unifmono} hold, and
let \(\epsi>0\) be fixed. 
Let \(K_+\) be the nonnegative-density cone defined in
\eqref{nncone}. Fix \(x_0\in K_+\), choose \(R\) as in
\eqref{eq:bounded_radius}, and set
\(K=K_+\cap\overline B_X(0,R)\).
Set
\[
\kappa=\max\{\bar\beta,\bar\gamma\},
\qquad
\kappa'=\frac{\kappa}{\kappa-1}.
\]
If \((\lambda_n)\subset(0,\infty)\) satisfies
\[
\sum_n\lambda_n=\infty,
\qquad
\sum_n\lambda_n^{\kappa'}<\infty,
\]
then the sequences generated by both the one-step iteration \eqref{eq:one_step_algorithm} and the
frozen two-step iteration \eqref{eq:variable_step_algorithm} on \(K\)
converge strongly in
\(L^{\bar\beta}(\mathbb T^d)\times W^{1,\bar\gamma}(\mathbb T^d)\) to the
unique solution \((m_\epsi,u_\epsi)\) of
\(\VI(\mathcal A_\epsi,K)\).
\end{teo}
Theorem~\ref{thm:main} extends the monotone approach of \cite{almullaTwoNumericalApproaches2017} from $L^2$ to $L^{\bar{\beta}} \times W^{1,\bar{\gamma}}$ by replacing the inner-product geometry with a mixed-exponent Bregman divergence. Moreover, it complements the existence and a priori theory of \cite{ferreiraSolvingMeanFieldGames2025} with a convergent iterative scheme in the same function-space norm.
The proof combines quantitative monotonicity of the regularized MFG operator,
its boundedness on the admissible set together with a mirror-increment
estimate, and a one-step Bregman inequality for the constrained mirror map.
The vanishing-regularization passage is inherited from
\cite{ferreiraSolvingMeanFieldGames2025}; we additionally prove strong
convergence in \(X\) along the selected subsequences under the quantitative
monotonicity assumption.

The convergence theorem is a fixed-\(\epsi\) result. The numerical experiments
in Section~\ref{sec:numerics} validate the discretization against manufactured
solutions and compare the implementations at matched effective step sizes:
when the ball constraint is inactive, the frozen two-step update coincides
with a one-step update of twice the step size, and the apparent advantage of
the two-step implementation over the nominal one-step baseline is a
step-size effect.

The paper is organized as follows. In Section~\ref{sec:assumptions}, we introduce the notation, exponent relations, and structural assumptions.
In Section~\ref{sec:reg_operator_properties}, we define the regularized operator,
prove quantitative monotonicity, recall the fixed-\(\eps\) well-posedness
and vanishing-regularization result, strengthen the latter to strong
convergence in \(X\) as \(\eps\to0\), and establish the boundedness of the
operator on the admissible set. In Section~\ref{sec:bregman}, we introduce the
Bregman geometry adapted to the mixed exponents. In Section~\ref{sec:algorithms}, we introduce the one-step iteration and its frozen two-step variant, and in 
Section~\ref{sec:convergence} prove the convergence of both. In
Section~\ref{sec:numerics}, we present one- and two-dimensional numerical
experiments: manufactured-solution validation, computational work under
refinement, and the effective-step equivalence of the implementations.
Finally, in Section~\ref{sec:conclusion}, we discuss limitations and future directions.

\section{Preliminaries}\label{sec:assumptions}

In this section, we set up the function spaces, exponents, and structural hypotheses used throughout the paper. The assumptions are organized into two layers, each playing a distinct role in the analysis. Assumption~\ref{ass:base} contains the structural hypotheses on $H_0$, $g$, and $V$ that suffice to define the regularized operator and to establish monotonicity and hemicontinuity. Assumption~\ref{ass:unifmono} strengthens the monotonicity of $g$ to a quantitative bound; this is what supplies coercivity in the density variable, uniqueness of the regularized solution, and the dissipation rate driving the iterative scheme. 

\subsection{Notation}
Throughout the paper, 
$\T:=\RR^d/\mathbb Z^d$ is the flat \(d\)-dimensional torus,
equipped with Lebesgue measure normalized by \(|\T|=1\).
We work with the exponents
\[
\alpha\ge 2,
\qquad
\beta\ge 1,
\qquad
\bbar:=\beta+1,
\qquad
\gbar:=\alpha\frac{\beta+1}{\beta}.
\]
These exponents determine the functional spaces for $m$ and $u$ in Problem~\ref{prob:main}.

The exponents $\bbar$ and $\gbar$ are chosen so that the Hamiltonian, coupling,
and transport flux have the integrability required by the duality pairings in
the variational formulation.
For \(m\in L^{\bbar}(\T)\), \(m\ge0\), and
\(u\in W^{1,\gbar}(\T)\), Assumption~\ref{ass:base} gives
\(H_0(x,Du),g(m)\in L^{\bbar'}(\T)\) and
\(mD_pH_0(x,Du)\in L^{\gbar'}(\T;\RR^d)\).
Writing $\bbar' = \frac{\beta+1}{\beta}$ and $\gbar' = \frac{\alpha(\beta+1)}{\alpha(\beta+1)-\beta}$ for the conjugate exponents of $\bbar$ and $\gbar$, we have the identities
\[
\frac{\alpha}{\gbar}=\frac1{\bbar'},
\qquad
\frac{\alpha-1}{\gbar}+\frac1{\bbar}=\frac1{\gbar'},
\]
which are used repeatedly in the H\"older estimates and in the Bregman bounds adapted to the mixed Lebesgue--Sobolev structure. The standing assumptions $\alpha\ge2$ and $\beta\ge1$ also imply the magnitude bounds
\[
\bbar\ge2,
\qquad
\bbar'>1,
\qquad
\gbar=\alpha\bbar'>2,
\qquad
\gbar'>1.
\]
The lower bounds $\bbar\ge2$ and $\gbar>2$ are what make the pointwise inequalities~\eqref{eq:power_mono_pointwise} and~\eqref{eq:p_lap_mono} available with $r=\bbar$ and $r=\gbar$, which is the form in which they enter the monotonicity and Bregman estimates.

For \(u\in W^{1,\bar\gamma}(\T)\), we use the norm
\[
\|u\|_{W^{1,\bar\gamma}(\T)}
:=
\|u\|_{L^{\bar\gamma}(\T)}+\|Du\|_{L^{\bar\gamma}(\T)}.
\]
Set
\[
X:=L^{\bbar}(\T)\times W^{1,\gbar}(\T)
\]
with the product norm
\[
\|(m,u)\|_X:=\|m\|_{L^{\bbar}}+\|u\|_{W^{1,\gbar}}.
\]
Since \(1<\bbar,\gbar<\infty\), the space \(X\) is reflexive. Its dual is
\[
X^*=L^{\bbar'}(\T)\times W^{-1,\gbar'}(\T)
\]
for which we use the equivalent norm
\[
\|(f,\ell)\|_{X^*}:=\|f\|_{L^{\bbar'}}+\|\ell\|_{W^{-1,\gbar'}}.
\]
The duality pairing is
\[
\pair{(f,\ell)}{(m,u)}_{X^*,X}
:=
\int_\T f\,m\,dx
+
\pair{\ell}{u}_{W^{-1,\gbar'},W^{1,\gbar}}.
\]

The admissible set for the variational inequality is
\begin{equation}
\label{nncone}
K_+:=
\left\{
(m,u)\in X:\;
m\ge0\ \text{a.e. in }\T
\right\}.
\end{equation}
The nonnegativity constraint \(m\ge0\) reflects the physical interpretation of \(m\) as a density of agents.
The set $K_+$ is nonempty, convex, and norm-closed in $X$, hence weakly closed. 
No mass constraint is imposed in \(K_+\). Thus, when the identity
\(\int_{\T}m\,dx=1\) holds for the unregularized limit equation, it is a
consequence of the transport equation rather than a constraint in the
admissible set.
For fixed \(\epsi>0\), the solution of the regularized problem may not have unit mass. 

Throughout the paper, \(C\) denotes a positive constant that may
change from line to line and depends only on the fixed structural constants,
the exponents, the domain, and the prescribed data norms, such as
\(\|V\|_{L^{\bbar'}}\). The constants \(C_R\) and \(C_{R,\eps}\) may
additionally depend on \(R\), and on \(R\) and \(\eps>0\), respectively. None of these constants depends on the particular elements being estimated.

\subsection{Elementary inequalities}
We record two pointwise estimates and the Bregman bounds they imply, all of
which are used repeatedly in
Sections~\ref{sec:reg_operator_properties}--\ref{sec:bregman}. 

For every \(r\ge2\) and all \(a,b\ge0\),
\begin{equation}\label{eq:power_mono_pointwise}
\bigl(a^{r-1}-b^{r-1}\bigr)(a-b)
\ge
|a-b|^r.
\end{equation}
For every \(r\ge2\) and all \(\xi,\eta\in\RR^N\),
\begin{equation}\label{eq:p_lap_mono}
\bigl(|\xi|^{r-2}\xi-|\eta|^{r-2}\eta\bigr)\cdot(\xi-\eta)
\ge
2^{2-r}|\xi-\eta|^r.
\end{equation}
Inequality~\eqref{eq:power_mono_pointwise} is the nonnegative scalar case of~\eqref{eq:p_lap_mono},
with the sharper constant \(1\) obtained by a direct one-dimensional argument.

\begin{lemma}[Power Bregman bounds]\label{lem:bregman_power}
Let \(r\ge2\). Define the scalar function \(\psi_r(s)=\frac1r s^r\) for \(s\ge0\), and the vector function \(\varphi_r(\xi)=\frac1r|\xi|^r\) for \(\xi\in\RR^N\).
Then, for all \(a,b\ge0\),
\begin{equation}\label{eq:scalar_bregman_power}
\psi_r(a)-\psi_r(b)-\psi_r'(b)(a-b)
\ge
\frac1r |a-b|^r.
\end{equation}
Furthermore, with \(c_r = \frac{2^{2-r}}{r}\), for all \(\xi,\eta\in\RR^N\),
\begin{equation}\label{eq:vector_bregman_power}
\varphi_r(\xi)-\varphi_r(\eta)-D\varphi_r(\eta)\cdot(\xi-\eta)
\ge
c_r|\xi-\eta|^r.
\end{equation}
\end{lemma}
\begin{remark}
The scalar estimate is stated on \([0,\infty)\), which is the only case used 
for the density variable on \(K_+\).
\end{remark}
\begin{proof}
Both bounds follow from the fundamental theorem of calculus and the respective pointwise inequalities.
For the scalar case, expressing the Bregman divergence in integral form yields
\[
\psi_r(a)-\psi_r(b)-\psi_r'(b)(a-b)
=
\int_0^1
\bigl(\psi_r'(b+t(a-b))-\psi_r'(b)\bigr)(a-b)\,dt.
\]
Since \(\psi_r'(s)=s^{r-1}\), applying \eqref{eq:power_mono_pointwise} bounds the integrand from below by \(t^{r-1}|a-b|^r\). Integrating over \(t\in[0,1]\) establishes \eqref{eq:scalar_bregman_power}.
The vector case \eqref{eq:vector_bregman_power} follows by an identical argument: replacing \(\psi_r\) with \(\varphi_r\) and using \eqref{eq:p_lap_mono} bounds the corresponding integrand by \(2^{2-r}t^{r-1}|\xi-\eta|^r\), and integration then yields \(c_r = 2^{2-r}/r\).
\end{proof}
We also use the following standard pointwise estimate for the power maps: for
every \(r\ge2\) there exists \(C_r>0\) such that
\begin{equation}\label{eq:Jp_lip_pointwise}
\bigl||a|^{r-2}a-|b|^{r-2}b\bigr|
\le
C_r\bigl(|a|^{r-2}+|b|^{r-2}\bigr)|a-b|
\end{equation}
for all \(a,b\in\RR^N\), with \(N=1\) in the scalar case
and zero powers interpreted as \(1\) when \(r=2\).

\subsection{Assumptions}

\begin{assumption}[Structural assumptions]\label{ass:base}
The functions \(H_0\), \(g\), and \(V\) satisfy the following conditions.
\begin{enumerate}[label=(A\arabic*)]
\item
$H_0:\T\times\RR^d\to\RR$ is measurable in $x$, continuously differentiable in $p$,
and for a.e.\ $x\in\T$ the map $p\mapsto H_0(x,p)$ is convex.
We also assume that \(D_pH_0\) is jointly measurable in \((x,p)\).

\item
There exists \(C>0\) such that, for a.e. \(x\in\T\) and all \(p\in\RR^d\),
\[
\frac1C |p|^\alpha-C
\le
H_0(x,p)
\le
C(1+|p|^\alpha).
\]

\item
There exists \(C>0\) such that, for a.e. \(x\in\T\) and all \(p\in\RR^d\),
\begin{equation}\label{eq:DpH_growth}
|D_pH_0(x,p)|
\le
C(1+|p|^{\alpha-1}).
\end{equation}

\item\label{it:g_basic}
\(g:[0,\infty)\to\RR\) is continuous, nondecreasing, and satisfies \(g(0)=0\).

\item\label{it:g_growth_assum}
There exists \(C>0\) such that, for all \(m\ge0\),
\begin{equation}\label{eq:g_growth}
|g(m)|
\le
C(1+m^{\bbar-1}).
\end{equation}

\item
\(V\in L^{\bbar'}(\T)\).
\end{enumerate}
\end{assumption}

Assumption~\ref{ass:base} does not, by itself, give coercivity in the \(m\)-variable. Indeed, conditions~\ref{it:g_basic}--\ref{it:g_growth_assum} allow $g$ to grow strictly slower than $m^{\bbar-1}$, in which case $\int_\T g(m)m\,dx$ does not control $\|m\|_{L^{\bbar}}^{\bbar}$ from below. The required density coercivity is instead supplied by Assumption~\ref{ass:unifmono} below.

\begin{assumption}[Quantitative monotonicity of the coupling]\label{ass:unifmono}
There exists \(c_g>0\) such that, for all \(m,\theta\ge0\),
\begin{equation}\label{eq:g_mono}
(g(m)-g(\theta))(m-\theta)
\ge
c_g |m-\theta|^{\bbar}.
\end{equation}
\end{assumption}

Thus, the convergence theory below applies to uniformly monotone power-type 
couplings and not to merely monotone couplings such as bounded or logarithmic 
interactions. Since \(g(0)=0\), taking \(\theta=0\) in \eqref{eq:g_mono} gives
\[
g(m)m\ge c_g m^{\bbar}
\qquad
\text{for all }m\ge0.
\]
This gives density coercivity; the full two-point estimate \eqref{eq:g_mono} is used for uniqueness in Theorem~\ref{thm:exist_limit} and for dissipation in Section~\ref{sec:convergence}.

\begin{example}[Power-type Hamiltonians]\label{ex:hamiltonians}
The assumptions above are satisfied by
\[
H_0(x,p)=a(x)|p|^\alpha+b(x)\cdot p+c_0(x),
\qquad
g(m)=c\,m^{\bbar-1},
\]
where \(a,c_0\in L^\infty(\T)\), \(b\in L^\infty(\T;\RR^d)\),
\(\operatorname*{ess\,inf}_{\T} a>0\), and \(c>0\)
(recall $\bbar=\beta+1$).
More generally, one may add a perturbation $\widetilde g$,
\[
g(m)=c\,m^{\bbar-1}+\widetilde g(m),
\]
provided $\widetilde g$ is continuous and nondecreasing,
$\widetilde g(0)=0$, and $|\widetilde g(m)|\le C(1+m^{\bbar-1})$ for all $m\ge0$
(for example $\widetilde g(m)=m^\nu$ with $0<\nu<1$, which is not locally
Lipschitz at $m=0$);
the principal term $c\,m^{\bbar-1}$ satisfies the quantitative
monotonicity~\eqref{eq:g_mono}, and the monotonicity of $\widetilde g$
preserves the lower bound. 
Indeed, by \eqref{eq:power_mono_pointwise},
\[
c\bigl(m^{\bar\beta-1}-\theta^{\bar\beta-1}\bigr)(m-\theta)
\ge c|m-\theta|^{\bar\beta}.
\]
The preceding example is in the form used in the
numerical experiments of Section~\ref{sec:numerics}.
\end{example}

\section{The regularized operator}\label{sec:reg_operator_properties}

In this section, we introduce the regularized MFG operator and record the properties
needed for the algorithmic analysis. 
We first define the operator on the unbounded nonnegative-density constraint
set \(K_+\), where the existence theory is naturally posed.
We then prove quantitative monotonicity, recall fixed-\(\eps\)
well-posedness and the vanishing-regularization limit, and strengthen
the convergence of the regularized solutions. Finally, we introduce
a bounded admissible set and establish the boundedness of the operator on it.

\subsection{Definition and variational formulation}

The original stationary MFG system corresponds to the formal operator
\[
\mathcal A_0(m,u)
=
\begin{bmatrix}
-u - H_0(x,Du) + g(m) + V(x)\\[4pt]
-\diver\!\big(m D_pH_0(x,Du)\big) + m - 1
\end{bmatrix},
\]
acting on \(K_+\).

The difficulty is that the unregularized operator does not provide the
value-function coercivity needed for the Banach-space variational theory.
In particular, Assumption~\ref{ass:unifmono} supplies density coercivity, 
whereas no comparable coercive term controls the full 
\(W^{1,\gbar}\)-norm of \(u\) without regularization.
 Following the low-order regularization framework of~\cite{ferreiraSolvingMeanFieldGames2025}, we therefore replace $\mathcal A_0$ by a coercive approximation $\mathcal A_\eps$.

For $\eps>0$, define $\mathcal A_\eps:K_+\to X^*$ by
\begin{equation}\label{eq:operator_duality}
\begin{aligned}
\pair{\mathcal A_\eps(m,u)}{(\mu,v)}
&:=
\int_\T \bigl[-u-H_0(x,Du)+g(m)+V(x)\bigr]\mu\,dx
\\
&\quad
+\int_\T m D_pH_0(x,Du)\cdot Dv\,dx
+\int_\T (m-1)v\,dx
\\
&\quad
+\eps\int_\T |Du|^{\gbar-2}Du\cdot Dv\,dx
+\eps\int_\T |u|^{\gbar-2}u\,v\,dx,
\end{aligned}
\end{equation}
for every $(\mu,v)\in X$. Assumption~\ref{ass:base} and the exponent identities in
Section~\ref{sec:assumptions} ensure that this pairing is well defined: \(H_0(x,Du)\), \(g(m)\), \(u\), and \(V\) belong to
\(L^{\bbar'}(\T)\), while
\(mD_pH_0(x,Du)\in L^{\gbar'}(\T;\RR^d)\), and \(m-1\) defines an element of
\(W^{-1,\gbar'}(\T)\).

Equivalently, in distributional form,
\begin{equation}\label{eq:operator}
\mathcal A_\eps\binom{m}{u}
=
\begin{bmatrix}
-u - H_0(x,Du) + g(m) + V(x)\\[6pt]
-\diver\!\big(m D_pH_0(x,Du)\big) + m - 1
-\eps\,\diver\!\big(|Du|^{\gbar-2}Du\big)
+\eps\,|u|^{\gbar-2}u
\end{bmatrix}.
\end{equation}

The two \(\eps\)-terms form the low-order \(W^{1,\gbar}\)-duality 
regularization. They are matched to the Banach-space geometry of $W^{1,\gbar}(\T)$ and contribute the coercive quantity
\[
\eps\int_\T |Du|^{\gbar}\,dx
+
\eps\int_\T |u|^{\gbar}\,dx
\]
to the energy estimates.

The regularized variational problem $\VI(\mathcal A_\eps,K_+)$ is:
find $(m_\eps,u_\eps)\in K_+$ such that
\begin{equation}\label{eq:VI_eps}
\pair{\mathcal A_\eps(m_\eps,u_\eps)}{(\mu,v)-(m_\eps,u_\eps)} \ge 0
\qquad
\text{for all } (\mu,v)\in K_+.
\end{equation}
When $\eps=0$, one recovers the original operator $\mathcal A_0$.

The existence theory and the vanishing-regularization limit for
\eqref{eq:VI_eps} are recalled in Theorem~\ref{thm:exist_limit} from
\cite{ferreiraSolvingMeanFieldGames2025}.

\subsection{Quantitative monotonicity}\label{sec:quant_mono}

We now record the quantitative monotonicity estimate used later for uniqueness
and for the dissipation in the convergence proof. The qualitative monotonicity,
hemicontinuity, and coercivity needed for existence are part of the
Banach-space theory recalled from~\cite{ferreiraSolvingMeanFieldGames2025}.

\begin{theorem}[Quantitative monotonicity of $\mathcal A_\eps$]\label{thm:Aeps_properties}
Suppose Assumptions~\ref{ass:base} and~\ref{ass:unifmono} hold. There exists
\(c>0\), independent of \(\eps\), such that for every \(\eps\ge0\)
and all \((m,u),(\theta,v)\in K_+\),
\begin{equation}\label{eq:unif_mono}
\pair{\mathcal A_\eps(m,u)-\mathcal A_\eps(\theta,v)}{(m-\theta,u-v)}
\ge
c\|m-\theta\|_{L^{\bbar}}^{\bbar}
+
c\eps\|u-v\|_{W^{1,\gbar}}^{\gbar}.
\end{equation}
\end{theorem}

\begin{proof}
Expanding the duality pairing, the cross-terms produced by the zero-order
differences \(-(u-v)\) and \((m-\theta)\) cancel exactly, leaving
\[
\begin{aligned}
&\pair{\mathcal A_\eps(m,u)-\mathcal A_\eps(\theta,v)}{(m-\theta,u-v)}
=
\int_\T (g(m)-g(\theta))(m-\theta)\,dx
\\
&\qquad
+\eps\int_\T
\bigl(|Du|^{\gbar-2}Du-|Dv|^{\gbar-2}Dv\bigr)\cdot(Du-Dv)\,dx
\\
&\qquad
+\eps\int_\T
\bigl(|u|^{\gbar-2}u-|v|^{\gbar-2}v\bigr)(u-v)\,dx
\\
&\qquad
+\mathcal R(m,\theta,u,v),
\end{aligned}
\]
where $\mathcal R$ collects the Hamilton--Jacobi and transport contributions,
\[
\begin{aligned}
\mathcal R(m,\theta,u,v)
&=
\int_\T m\bigl[H_0(x,Dv)-H_0(x,Du)-D_pH_0(x,Du)\cdot(Dv-Du)\bigr]\,dx
\\
&\quad
+\int_\T \theta\bigl[H_0(x,Du)-H_0(x,Dv)-D_pH_0(x,Dv)\cdot(Du-Dv)\bigr]\,dx.
\end{aligned}
\]
By Assumption~\ref{ass:base}, each bracket is the Bregman divergence of the convex function 
\(p\mapsto H_0(x,p)\), evaluated in opposite directions.
The brackets are nonnegative by convexity; since $m,\theta\ge0$ a.e., we obtain $\mathcal R\ge0$. 

By Assumption~\ref{ass:unifmono},
\[
\int_\T (g(m)-g(\theta))(m-\theta)\,dx
\ge
c_g\|m-\theta\|_{L^{\bbar}}^{\bbar}.
\]
Using~\eqref{eq:p_lap_mono} with \(r=\gbar\), the two \(\eps\)-terms are
bounded below by
\[
2^{2-\gbar}\eps\|Du-Dv\|_{L^{\gbar}}^{\gbar}
\qquad\text{and}\qquad
2^{2-\gbar}\eps\|u-v\|_{L^{\gbar}}^{\gbar}.
\]
Combining these bounds and using
\(a^{\gbar}+b^{\gbar}\ge 2^{1-\gbar}(a+b)^{\gbar}\), we obtain
\eqref{eq:unif_mono} with \(c=\min\{c_g,\,2^{3-2\gbar}\}\).
\end{proof}
This estimate is the source of the fixed-\(\eps\) error control in
Section~\ref{sec:convergence}: the value-function error is dissipated
through the \(\eps\)-weighted term, independently of the particular
mirror iteration used.

\subsection{Existence, uniqueness, and the limit \texorpdfstring{$\eps\to0$}{epsilon -> 0}}\label{sec:existence}

We now recall the well-posedness of \(\VI(\mathcal A_\eps,K_+)\) and the
vanishing-regularization limit. Existence, the uniform a priori bound, and the
identification of the limit are imported from
\cite{ferreiraSolvingMeanFieldGames2025}; uniqueness follows from the
quantitative monotonicity estimate \eqref{eq:unif_mono}.

\begin{theorem}[Well-posedness and vanishing-regularization limit]\label{thm:exist_limit}
Suppose Assumptions~\ref{ass:base} and~\ref{ass:unifmono} hold.
\begin{enumerate}[label=\textup{(\arabic*)}]
\item For every $\eps>0$, there exists a unique $z_\eps=(m_\eps,u_\eps)\in K_+$ such that
\begin{equation}\label{eq:VI_eps_strong}
\pair{\mathcal A_\eps(z_\eps)}{z-z_\eps}\ge0
\qquad
\text{for all }z\in K_+,
\end{equation}
and
\begin{equation}\label{eq:apriori_bound}
\|m_\eps\|_{L^{\bbar}} + \|u_\eps\|_{W^{1,\gbar}} \le C
\end{equation}
for a constant $C>0$ independent of $\eps$.
\item Let $\eps_k\downarrow0$. Up to a subsequence, $z_{\eps_k}\rightharpoonup z=(m,u)$ weakly in $X$, with $z\in K_+$, and $(m,u)$ is a strong solution of the stationary MFG system in the sense that
\begin{equation}\label{eq:mfg_original}
\begin{cases}
u+H_0(x,Du)\le g(m)+V(x) &\text{a.e. in }\T,\\[4pt]
u+H_0(x,Du)=g(m)+V(x) &\text{a.e. on }\{m>0\},\\[4pt]
-\diver\!\bigl(mD_pH_0(x,Du)\bigr)+m=1 &\text{in }\mathcal D'(\T).
\end{cases}
\end{equation}
\end{enumerate}
\end{theorem}
\begin{proof}

The separable
Hamiltonian
\[
H(x,p,m)=H_0(x,p)-g(m)
\]
satisfies the hypotheses of the power-growth theory in
\cite[Assumptions~2.2 and~2.4]{ferreiraSolvingMeanFieldGames2025}.

Indeed, the standing measurability and continuity assumptions on \(H_0\) and
\(g\) give the regularity conditions \((\mathrm H1)\)--\((\mathrm H3)\) used
there. Moreover, for \(m>0\), Assumption~\ref{ass:unifmono} with
\(\theta=0\) gives
\[
g(m)\ge c_g m^{\bar\beta-1}=c_g m^\beta,
\]
while Assumption~\ref{ass:base} gives
\[
|g(m)|\le C(1+m^\beta).
\]
Consequently,
\[
H(x,0,m)\le C-c_gm^\beta,
\]
\[
|D_pH(x,p,m)|=|D_pH_0(x,p)|
\le C(1+|p|^{\alpha-1})
\le C(1+|p|^{\alpha-1}+m^{\beta-\beta/\alpha}),
\]
and
\[
H(x,p,m)
=
H_0(x,p)-g(m)
\ge
\frac1C|p|^\alpha-C(1+m^\beta).
\]
Thus the power-growth bounds in
\cite[Assumption~2.4]{ferreiraSolvingMeanFieldGames2025} hold.

The monotonicity condition \cite[Assumption~2.2]{ferreiraSolvingMeanFieldGames2025}
is the pointwise inequality obtained by expanding
\(\pair{\mathcal A_\eps(m,u)-\mathcal A_\eps(\theta,v)}{(m-\theta,u-v)}\)
as in the proof of Theorem~\ref{thm:Aeps_properties}: the coupling term is
nonnegative by the monotonicity of \(g\), and the two Hamiltonian brackets are
nonnegative by the convexity of \(p\mapsto H_0(x,p)\).

With this identification, the operator in
\cite[(3.1)]{ferreiraSolvingMeanFieldGames2025} coincides with the operator
\(\mathcal A_\epsi\) defined above. Therefore, for each fixed
\(\epsi>0\), existence of a regularized solution follows from
\cite[Proposition~3.3]{ferreiraSolvingMeanFieldGames2025}. The uniform
\(\epsi\)-independent estimate follows from
\cite[Theorem~3.5]{ferreiraSolvingMeanFieldGames2025}. Finally, the
vanishing-regularization passage and the identification of the limiting strong
solution follow from the limit argument in
\cite[Section~3]{ferreiraSolvingMeanFieldGames2025}, together with the
strong-solution property in
\cite[Lemma~3.4]{ferreiraSolvingMeanFieldGames2025}.
For the limiting unregularized solution \(z=(m,u)\), testing the transport
equation in \eqref{eq:mfg_original} with the constant function
\(\varphi\equiv1\) gives
\[
\int_{\T}m\,dx=1.
\]
This mass identity is not imposed in \(K_+\), and it need not hold for the
regularized solution \(z_\epsi\) at fixed \(\epsi>0\).

It remains to prove uniqueness for fixed $\eps>0$. Suppose $z_\eps^1,z_\eps^2\in K_+$ both satisfy~\eqref{eq:VI_eps_strong}. Testing each inequality against the other solution and adding yields
\[
\pair{\mathcal A_\eps(z_\eps^1)-\mathcal A_\eps(z_\eps^2)}{z_\eps^1-z_\eps^2}\le0,
\]
and Theorem~\ref{thm:Aeps_properties} forces $c\|m_\eps^1-m_\eps^2\|_{L^{\bbar}}^{\bbar}+c\eps\|u_\eps^1-u_\eps^2\|_{W^{1,\gbar}}^{\gbar}\le0$, so $z_\eps^1=z_\eps^2$.
\end{proof}

This argument gives uniqueness only for fixed \(\eps>0\); uniqueness of 
unregularized solutions is not asserted here.

\subsection{Strong convergence in the vanishing-regularization limit}\label{sec:strong_convergence}

Theorem~\ref{thm:exist_limit} identifies weak limits of the regularized
solutions as strong solutions of the original stationary MFG system. The next
result refines this limit passage. Although it is not used in the fixed-\(\eps\)
convergence analysis of the mirror method, it shows that the regularization is
compatible with the norm of \(X\): along every weakly convergent
vanishing-regularization sequence, the solutions converge strongly in
\(X\), and the Hamiltonian Bregman defect for the gradients,
weighted by the limiting density, vanishes.

\begin{theorem}[Strong convergence and vanishing Hamiltonian defect]
\label{thm:strong_limit}
Under Assumptions~\ref{ass:base} and~\ref{ass:unifmono}, let \(\eps_k\downarrow0\), and let
\(z_{\eps_k}=(m_{\eps_k},u_{\eps_k})\) be a subsequence of regularized
solutions such that
\[
z_{\eps_k}\rightharpoonup z=(m,u)
\qquad\text{weakly in }X,
\]
as provided by Theorem~\ref{thm:exist_limit}.
Then:
\begin{enumerate}[label=\textup{(\arabic*)}]
\item \(z_{\eps_k}\to z\) strongly in \(X\).
\item The gradient of the value function satisfies
\begin{equation}\label{eq:grad_conv}
\lim_{k\to\infty} \int_\T m\bigl[H_0(x,Du_{\eps_k})-H_0(x,Du)-D_pH_0(x,Du)\cdot(Du_{\eps_k}-Du)\bigr]\,dx=0.
\end{equation}
\end{enumerate}
\end{theorem}

\begin{proof}
Let \(w=(w_m,w_u)\in K_+\), and set
\(F_1:=-u-H_0(x,Du)+g(m)+V\) and
\(F_2:=-\diver\bigl(mD_pH_0(x,Du)\bigr)+m-1\in W^{-1,\gbar'}(\T)\).
Since \(z=(m,u)\) is a strong solution of \eqref{eq:mfg_original}, \(F_2=0\)
in \(W^{-1,\gbar'}(\T)\), \(F_1\ge0\) a.e., and \(F_1m=0\) a.e. Hence
\[
\pair{\mathcal A_0(z)}{w-z}
=
\int_\T F_1(w_m-m)\,dx
=
\int_\T F_1w_m\,dx
\ge0
\qquad\text{for all }w\in K_+,
\]
that is, \(z\) solves the unregularized variational inequality.

Taking \(w=z_{\eps_k}\in K_+\) gives
\[
\pair{\mathcal A_0(z)}{z_{\eps_k}-z}\ge0.
\]

Because \(z_{\eps_k}\) solves \(\VI(\mathcal A_{\eps_k},K_+)\), taking \(w=z\) yields \(\pair{\mathcal A_{\eps_k}(z_{\eps_k})}{z-z_{\eps_k}}\ge0\).
Adding these inequalities, we obtain
\[
\pair{\mathcal A_{\eps_k}(z_{\eps_k})-\mathcal A_0(z)}{z_{\eps_k}-z}\le 0.
\]
We decompose \(\mathcal A_{\eps_k}(z)=\mathcal A_0(z)+\eps_k \mathcal J(z)\), 
where \(\mathcal J(z)\in X^*\) collects the \(\gbar\)-Laplacian regularization terms in \eqref{eq:operator_duality} and is defined by
\(\pair{\mathcal J(z)}{(\mu,v)}=\int_\T |Du|^{\gbar-2}Du\cdot Dv\,dx+\int_\T |u|^{\gbar-2}u\,v\,dx\).
Substituting this decomposition into the previous inequality gives
\[
\pair{\mathcal A_{\eps_k}(z_{\eps_k})-\mathcal A_{\eps_k}(z)}{z_{\eps_k}-z}
\le
-\eps_k\pair{\mathcal J(z)}{z_{\eps_k}-z}.
\]
By the monotonicity of \(\mathcal A_{\eps_k}\), the left-hand side is
nonnegative. Hence
\[
0\le
\pair{\mathcal A_{\eps_k}(z_{\eps_k})-\mathcal A_{\eps_k}(z)}
{z_{\eps_k}-z}
\le
\eps_k
\left|
\pair{\mathcal J(z)}{z_{\eps_k}-z}
\right|.
\]
Because \(z_{\eps_k}\rightharpoonup z\) weakly in \(X\) and \(\mathcal J(z)\in X^*\) is a fixed element of the dual space, the duality pairing itself tends to zero.

Expanding the term \(\pair{\mathcal A_{\eps_k}(z_{\eps_k})-\mathcal A_{\eps_k}(z)}{z_{\eps_k}-z}\) exactly as in the proof of Theorem~\ref{thm:Aeps_properties}, and retaining the quantitative lower bound for the regularization terms, we obtain
\[
c\|m_{\eps_k}-m\|_{L^{\bbar}}^{\bbar}
+
c\eps_k\|u_{\eps_k}-u\|_{W^{1,\gbar}}^{\gbar}
+
\mathcal R(m_{\eps_k},m,u_{\eps_k},u)
\le
\eps_k
\left|
\pair{\mathcal J(z)}{z_{\eps_k}-z}
\right|,
\]
where \(c>0\) is the constant from \eqref{eq:unif_mono} and \(\mathcal R\ge 0\) is the sum of the non-negative Hamiltonian cross-terms:
\[
\begin{aligned}
\mathcal R(m_{\eps_k},m,u_{\eps_k},u)
&=
\int_\T m_{\eps_k}\bigl[H_0(x,Du)-H_0(x,Du_{\eps_k})-D_pH_0(x,Du_{\eps_k})\cdot(Du-Du_{\eps_k})\bigr]\,dx
\\
&\quad
+\int_\T m\bigl[H_0(x,Du_{\eps_k})-H_0(x,Du)-D_pH_0(x,Du)\cdot(Du_{\eps_k}-Du)\bigr]\,dx.
\end{aligned}
\]
Since the right-hand side tends to zero and all three terms on the left-hand side are nonnegative, each converges to zero. The first gives \(m_{\eps_k}\to m\) strongly in \(L^{\bbar}(\T)\). We also obtain \(\mathcal R\to0\); as \(\mathcal R\) is a sum of two nonnegative integrals (by the convexity of \(H_0\)), each integral converges to zero, which proves \eqref{eq:grad_conv}.

Finally, dividing the preceding estimate by \(\eps_k>0\) and
discarding the other nonnegative terms gives
\[
c\|u_{\eps_k}-u\|_{W^{1,\gbar}}^{\gbar}
\le
\left|\pair{\mathcal J(z)}{z_{\eps_k}-z}\right|
\to0.
\]
Together with the density convergence, this proves strong convergence in \(X\).
\end{proof}
\begin{remark}[A density regularization rate]\label{rem:density_rate}
The comparison inequality in the proof holds for every \(\eps>0\), with
\(w=z\) in \(\VI(\mathcal A_\eps,K_+)\) and \(w=z_\eps\) in the
unregularized inequality:
\[
c\|m_\eps-m\|_{L^{\bbar}}^{\bbar}
\le
\eps\bigl|\pair{\mathcal J(z)}{z_\eps-z}\bigr|
\le
2M\eps\|\mathcal J(z)\|_{X^*},
\]
where \(M\) bounds \(\|z_\eps\|_X\) and \(\|z\|_X\) by
\eqref{eq:apriori_bound}. Hence \(\|m_\eps-m\|_{L^{\bbar}}\le C\eps^{1/\bbar}\),
and along the weakly convergent sequence of Theorem~\ref{thm:strong_limit}
the pairing tends to zero, so
\(\|m_{\eps_k}-m\|_{L^{\bbar}}=o(\eps_k^{1/\bbar})\).
\end{remark}

\subsection{The bounded admissible set for the algorithm}

We now turn back to the fixed-\(\eps\) problem, which is the setting for
the mirror algorithm. The convergence proof below is carried out on a bounded
admissible subset of \(K_+\).

 Let \(M>0\) be such that
\[
\|z_\eps\|_X\le M
\qquad
\text{for all }\eps>0,
\]
as guaranteed by \eqref{eq:apriori_bound}. Given an initial point
\(x_0\in K_+\), choose
\begin{equation}\label{eq:bounded_radius}
R>\max\{M,\|x_0\|_X\},
\end{equation}
and set
\begin{equation}\label{eq:bounded_K}
K:=K_+\cap\overline B_X(0,R).
\end{equation}
Thus \(x_0\in K\), and \(z_\eps\in K\) for every \(\eps>0\).

\begin{lemma}[Bounded admissible set]\label{lem:bounded_admissible_set}
Under Assumptions~\ref{ass:base} and~\ref{ass:unifmono},
the set \(K\) defined by \eqref{eq:bounded_K} is nonempty, convex,
norm-closed, bounded, and weakly compact. Moreover, for every \(\eps>0\),
the point \(z_\eps\) is the unique solution of
\(\VI(\mathcal A_\eps,K)\).
\end{lemma}

\begin{proof}
The set \(K\) is nonempty because \(x_0\in K_+\) and
\(\|x_0\|_X<R\), hence \(x_0\in K\). It is convex, norm-closed, and
bounded by construction. Since \(K_+\) is weakly closed and \(X\) is reflexive,
\(K\), as a weakly closed subset of the weakly compact ball
\(\overline B_X(0,R)\), is weakly compact.

Because \(K\subset K_+\) and \(z_\eps\) solves
\(\VI(\mathcal A_\eps,K_+)\), we have
\[
\pair{\mathcal A_\eps(z_\eps)}{z-z_\eps}\ge0
\qquad
\text{for all }z\in K.
\]
Thus \(z_\eps\) solves \(\VI(\mathcal A_\eps,K)\).

Let \(w\in K\) be another solution of \(\VI(\mathcal A_\eps,K)\). Testing the
inequality for \(z_\eps\) with \(w\), and the inequality for \(w\) with
\(z_\eps\), gives
\[
\pair{\mathcal A_\eps(w)-\mathcal A_\eps(z_\eps)}{w-z_\eps}\le0.
\]
On the other hand, \eqref{eq:unif_mono} implies
\[
\pair{\mathcal A_\eps(w)-\mathcal A_\eps(z_\eps)}{w-z_\eps}
\ge
c\|w_m-m_\eps\|_{L^{\bbar}}^{\bbar}
+
c\eps\|w_u-u_\eps\|_{W^{1,\gbar}}^{\gbar},
\]
where \(w=(w_m,w_u)\). Hence \(w=z_\eps\). Therefore \(z_\eps\) is the unique
solution of \(\VI(\mathcal A_\eps,K)\).
\end{proof}

From this point on, \(K\) denotes the bounded admissible set in
\eqref{eq:bounded_K}, while \(K_+\) denotes the original unbounded nonnegative-density constraint
set.
The bounded truncation is used in the infinite-dimensional convergence proof to
obtain weak compactness for the mirror minimization and to work with bounded-set
estimates, including the operator bound and the Bregman comparison bounds. In the
finite-dimensional experiments of Section~\ref{sec:numerics}, the analogous
truncation is monitored and remains inactive.

\subsection{Boundedness of the operator on the admissible set}\label{sec:lipschitz}

The convergence proof in Section~\ref{sec:convergence} uses only the
following bound on the admissible set \(K\) fixed in \eqref{eq:bounded_K};
no continuity modulus of \(\mathcal A_\eps\) is required.

\begin{lemma}[Boundedness of \(\mathcal A_\eps\) on \(K\)]\label{lem:op_bound}
Under Assumption~\ref{ass:base}, for every \(R>0\) there exists \(C_R>0\)
such that, for all \(\eps>0\) and all
\((m,u)\in K=K_+\cap\overline B_X(0,R)\),
\begin{equation}\label{eq:op_bound}
\|\mathcal A_\eps(m,u)\|_{X^*}\le C_R(1+\eps).
\end{equation}
\end{lemma}
\begin{proof}
Write \(\mathcal A_\eps(m,u)=(F_1,F_2)\) as in \eqref{eq:operator}, with
\(\|m\|_{L^{\bbar}}+\|u\|_{W^{1,\gbar}}\le R\).

For the first component, the growth bounds in Assumption~\ref{ass:base} and
the identities \(\alpha\bbar'=\gbar\) and \((\bbar-1)\bbar'=\bbar\) give
\[
\|H_0(x,Du)\|_{L^{\bbar'}}\le C\bigl(1+\|Du\|_{L^{\gbar}}^{\alpha}\bigr),
\qquad
\|g(m)\|_{L^{\bbar'}}\le C\bigl(1+\|m\|_{L^{\bbar}}^{\bbar-1}\bigr),
\]
while \(\|u\|_{L^{\bbar'}}\le C\|u\|_{L^{\gbar}}\) because \(\gbar\ge\bbar'\)
and \(|\T|=1\), and \(V\in L^{\bbar'}(\T)\). Hence
\(\|F_1\|_{L^{\bbar'}}\le C_R\).

For the second component, we use
\(\|\diver G\|_{W^{-1,\gbar'}}\le\|G\|_{L^{\gbar'}}\) for
\(G\in L^{\gbar'}(\T;\RR^d)\) and \(\|h\|_{W^{-1,\gbar'}}\le C\|h\|_{L^{\bbar}}\)
for \(h\in L^{\bbar}(\T)\), the latter because
\(\|\varphi\|_{L^{\bbar'}}\le C\|\varphi\|_{W^{1,\gbar}}\).
By \eqref{eq:DpH_growth} and H\"older's inequality with
\(\frac1{\bbar}+\frac{\alpha-1}{\gbar}=\frac1{\gbar'}\),
\[
\|mD_pH_0(x,Du)\|_{L^{\gbar'}}
\le
C\|m\|_{L^{\bbar}}\bigl(1+\|Du\|_{L^{\gbar}}^{\alpha-1}\bigr)
\le C_R,
\]
and \(\|m-1\|_{W^{-1,\gbar'}}\le C(1+\|m\|_{L^{\bbar}})\). The two
regularization terms satisfy
\(\||Du|^{\gbar-2}Du\|_{L^{\gbar'}}=\|Du\|_{L^{\gbar}}^{\gbar-1}\) and
\(\||u|^{\gbar-2}u\|_{L^{\gbar'}}=\|u\|_{L^{\gbar}}^{\gbar-1}\), both bounded
by \(R^{\gbar-1}\). Hence \(\|F_2\|_{W^{-1,\gbar'}}\le C_R(1+\eps)\), and
\eqref{eq:op_bound} follows.
\end{proof}

The bound uses only the growth conditions of Assumption~\ref{ass:base} and
the exponent identities of Section~\ref{sec:assumptions}; it is uniform for
\(0<\eps\le1\).

\section{Mirror geometry}\label{sec:bregman}

In a Hilbert space \(H\), the basic unconstrained dynamics associated with a
monotone operator \(\mathcal A:H\to H\) is the flow
\[
\dot z = -\mathcal A(z),
\]
for which the squared distance \(\frac12\|z-z^*\|_H^2\) to a zero
\(z^*\) of \(\mathcal A\) is nonincreasing under monotonicity.
In our setting, $\mathcal A_\eps$ is defined on a subset of a Banach space $X$ and takes values in $X^*$ rather than $X$. Therefore, the expression $\dot z = -\mathcal A_\eps(z)$ does not define an evolution on $X$ and, in general, the Banach space norm may not be a Lyapunov functional.
In the Banach space setting, the mirror transformation replaces the primal flow
by a dual evolution, and the Bregman divergence replaces the Hilbert-space
Lyapunov functional.
In the first two
subsections, we develop the abstract mirror geometry and constrained mirror
step; in the last two, we adapt the construction to the mixed
Lebesgue--Sobolev exponents of our problem.

\subsection{Mirror transformation and Bregman divergence}

Let \(X\) be the reflexive Banach space fixed above, with dual \(X^*\) and
duality pairing \(\pair{\cdot}{\cdot}_{X^*,X}\). Let \(K\subset X\) be
nonempty, convex, and weakly closed, and let
\(\mathcal A:K\to X^*\) be a monotone operator.
The construction in this subsection is abstract and does not use the concrete
product structure of \(X\). The adaptation to
\(L^{\bbar}(\T)\times W^{1,\gbar}(\T)\) is carried out in Section~\ref{memm}.

The starting point of the construction is a fixed strictly convex potential
\(\Phi:X\to\Rr\). 
The mirror geometry depends on this choice, and different potentials give rise to different geometries on the same underlying space. 

\begin{definition}[Mirror potential]\label{def:mirror_potential}
A mirror potential on \(X\) is a Fréchet differentiable, strictly convex,
lower semicontinuous, and supercoercive functional \(\Phi:X\to\Rr\), meaning
\[
\frac{\Phi(z)}{\|z\|_X}\to\infty
\qquad\text{as }\|z\|_X\to\infty,
\]
whose derivative \(D\Phi:X\to X^*\) is bounded on bounded sets.
\end{definition}

The map $D\Phi:X\to X^*$ induced by the potential serves as a nonlinear identification between primal and dual variables. In a Hilbert space with $\Phi(z)=\tfrac12\|z\|^2$, $D\Phi$ is the canonical Riesz isomorphism; for general $\Phi$, it is a nonlinear duality map adapted to the geometry of $\Phi$.

The abstract mirror flow associated with $\mathcal A$ is the dual evolution
\begin{equation}\label{eq:mirror_flow_dual}
\frac{d}{ds}D\Phi(z(s))=-\mathcal A(z(s)).
\end{equation}

We now define the Bregman divergence generated by \(\Phi\), which replaces
the squared Hilbert distance in the discrete analysis.

\begin{definition}[Bregman divergence]\label{def:bregman}
The \emph{Bregman divergence} of a mirror potential $\Phi$ is
\begin{equation}\label{eq:bregman_def}
\mathcal D_\Phi(x,y):=\Phi(x)-\Phi(y)-\pair{D\Phi(y)}{x-y},
\qquad x,y\in X.
\end{equation}
\end{definition}

The strict convexity of $\Phi$ implies $\mathcal D_\Phi(x,y)\ge 0$ with equality if and only if $x=y$. The divergence is generally asymmetric, $\mathcal D_\Phi(x,y)\ne \mathcal D_\Phi(y,x)$. When $\Phi(z)=\tfrac12\|z\|^2$ on a Hilbert space, $\mathcal D_\Phi(x,y)=\tfrac12\|x-y\|^2$, and one recovers the squared Hilbert distance.

We conclude this subsection by recording the three-point identity for the Bregman
divergence. This identity is used repeatedly in the discrete analysis of
Section~\ref{sec:convergence}: for any \(a,b,c\in X\),
\begin{equation}\label{eq:three_point}
\mathcal D_\Phi(a,c)-\mathcal D_\Phi(a,b)-\mathcal D_\Phi(b,c)=\pair{D\Phi(b)-D\Phi(c)}{a-b},
\end{equation}
which follows directly from~\eqref{eq:bregman_def} by expanding both sides and cancelling. 

\subsection{Discrete mirror update and constrained mirror step}\label{sec:mirror_step}
We now apply the one-step mirror update to the bounded admissible set
\(K\) fixed in \eqref{eq:bounded_K}.
To transform the continuous mirror
flow~\eqref{eq:mirror_flow_dual} into a numerical algorithm, fix \(x\in K\) and a step size \(\tau>0\), and
freeze the operator at \(x\) over this step.
The unconstrained endpoint of an explicit Euler step in dual
variables, denoted by \(\widetilde y\), satisfies
\[
\frac{D\Phi(\widetilde y)-D\Phi(x)}{\tau}
=
-\mathcal A(x).
\]
Equivalently, setting \(\xi=\tau\mathcal A(x)\in X^*\), the dual update is
\[
D\Phi(\widetilde y)=D\Phi(x)-\xi.
\]
For every \(\eta\in X^*\), the functional
\(y\mapsto\Phi(y)-\pair{\eta}{y}\) has a unique minimizer,
by reflexivity, weak lower semicontinuity, supercoercivity,
and strict convexity. Its optimality condition is \(D\Phi(y)=\eta\).
Thus \(D\Phi:X\to X^*\) is bijective, and the dual update defines
the unique primal point
\begin{equation}\label{eq:mirror_free}
\widetilde y=(D\Phi)^{-1}\bigl(D\Phi(x)-\xi\bigr),
\end{equation}
which we call the \emph{free mirror translation} from \(x\) along \(\xi\).

Equivalently, the free translation is the unconstrained
minimizer of the Bregman-regularized functional
\[
y\mapsto \pair{\xi}{y}+\mathcal D_\Phi(y,x),
\]
since the first-order condition for this functional is
\(D\Phi(y)-D\Phi(x)+\xi=0\). However, the free step
\eqref{eq:mirror_free} may not lie in \(K\): even when \(x\in K\),
\(\widetilde y\) may violate the nonnegativity constraint or the bounded
truncation. To obtain an iteration that remains in \(K\), we minimize the same
functional over \(K\).

\begin{definition}[Mirror map]\label{def:mirror_map}
For \(x\in K\) and \(\xi\in X^*\), the \emph{mirror map} is
\begin{equation}\label{eq:mirror_step}
\Mirr_x(\xi):=\operatorname*{argmin}_{y\in K}\bigl\{\pair{\xi}{y}+\mathcal D_\Phi(y,x)\bigr\}.
\end{equation}
\end{definition}

Thus the constrained explicit Euler step with step size \(\tau\) is
\[
z^+=\Mirr_x\bigl(\tau\mathcal A(x)\bigr).
\]
Equivalently, \(z^+\in K\) is characterized by
\[
\pair{D\Phi(z^+)-D\Phi(x)+\tau\mathcal A(x)}{y-z^+}\ge0
\qquad
\text{for all }y\in K.
\]

We collect the basic properties of the mirror map.
\begin{lemma}[Properties of the mirror map]\label{lem:mirror_step_props}
Let \(\Phi\) be a mirror potential, let \(K\) be the bounded admissible set
defined in \eqref{eq:bounded_K}, let \(x\in K\), and let \(\xi\in X^*\). Then:
\begin{enumerate}[label=\textup{(\roman*)}]
\item The minimization in~\eqref{eq:mirror_step} has a unique solution, so
\(\Mirr_x(\xi)\) is well defined.
\item A point \(z^+\in K\) equals \(\Mirr_x(\xi)\) if and only if
\begin{equation}\label{eq:projection_characterization}
\pair{D\Phi(z^+)-D\Phi(x)+\xi}{y-z^+}\ge 0 \qquad \text{for all } y\in K.
\end{equation}
\item If the free translation \(\widetilde y\) in~\eqref{eq:mirror_free} lies in \(K\), then \(\Mirr_x(\xi)=\widetilde y\).
\item The point \(\Mirr_x(\xi)\) is the Bregman projection of
\(\widetilde y\) onto \(K\):
\begin{equation}\label{eq:mirror_as_projection}
\Mirr_x(\xi)=\operatorname*{argmin}_{y\in K}\mathcal D_\Phi(y,\widetilde y).
\end{equation}
\end{enumerate}
\end{lemma}
\begin{proof}
\emph{(i)}
Existence follows from the direct method: by
Lemma~\ref{lem:bounded_admissible_set}, \(K\) is weakly compact. Since
\(\Phi\) is convex and norm lower semicontinuous, it is weakly lower
semicontinuous; hence the objective
\(y\mapsto\pair{\xi}{y}+\mathcal D_\Phi(y,x)\) is weakly lower semicontinuous,
as it differs from \(\Phi(y)\) by an affine weakly continuous term.
Uniqueness follows from the strict convexity of
\(\mathcal D_\Phi(\cdot,x)\), inherited from the strict convexity of \(\Phi\).

\emph{(ii)} The variational characterization is the first-order optimality
condition for minimizing the differentiable convex objective
\(y\mapsto\pair{\xi}{y}+\mathcal D_\Phi(y,x)\) over the closed convex set \(K\).

\emph{(iii)} By~\eqref{eq:mirror_free},
\(D\Phi(\widetilde y)=D\Phi(x)-\xi\), so \(\widetilde y\) satisfies the
first-order optimality condition for minimizing
\(y\mapsto\pair{\xi}{y}+\mathcal D_\Phi(y,x)\) over \(X\). Hence it is the
unique unconstrained minimizer. If \(\widetilde y\in K\), this unconstrained
minimum is admissible and therefore agrees with the constrained minimum.

\emph{(iv)} By~\eqref{eq:mirror_free} we have \(D\Phi(\widetilde y)=D\Phi(x)-\xi\).
Substituting into~\eqref{eq:bregman_def}, with \(y\) in the first argument and
\(\widetilde y\) in the second,
\[
\mathcal D_\Phi(y,\widetilde y)=\Phi(y)-\Phi(\widetilde y)-\pair{D\Phi(x)-\xi}{y-\widetilde y}.
\]
Expanding and using the definition of \(\mathcal D_\Phi(y,x)\) gives
\[
\mathcal D_\Phi(y,\widetilde y)=\mathcal D_\Phi(y,x)+\pair{\xi}{y}+C(x,\xi),
\]
where \(C(x,\xi)\) is independent of \(y\). Minimizing both sides over \(y\in K\)
yields~\eqref{eq:mirror_as_projection}.
\end{proof}

The finite-dimensional implementation in Section~\ref{sec:numerics} uses the 
same mirror structure but enforces only the nonnegativity constraint, because 
the bounded truncation remains inactive in the reported runs.

\subsection{The mixed-exponent mirror potential}
\label{memm}

We now adapt the mirror geometry to the product space
\[
X=L^{\bbar}(\T)\times W^{1,\gbar}(\T).
\]
Define
\begin{equation}\label{eq:phi}
\Phi(m,u)
=
\frac1{\bbar}\int_\T |m|^{\bbar}\,dx
+
\frac1{\gbar}\int_\T
\bigl(|Du|^{\gbar}+|u|^{\gbar}\bigr)\,dx.
\end{equation}
On \(K\), where \(m\ge0\), the first term in \eqref{eq:phi} is simply
\(\frac1{\bbar}\int_\T m^{\bbar}\,dx\). We write
\[
\Phi(m,u)=\Phi_m(m)+\Phi_u(u),
\]
where
\[
\Phi_m(m)=\frac1{\bbar}\int_\T |m|^{\bbar}\,dx,
\qquad
\Phi_u(u)=
\frac1{\gbar}\int_\T
\bigl(|Du|^{\gbar}+|u|^{\gbar}\bigr)\,dx.
\]
Because \(\Phi\) is separable, the associated Bregman divergence splits:
\begin{equation}\label{eq:bregman_split}
\mathcal D_\Phi\bigl((m_1,u_1),(m_2,u_2)\bigr)
=
\mathcal D_{\Phi_m}(m_1,m_2)
+
\mathcal D_{\Phi_u}(u_1,u_2).
\end{equation}
The split \eqref{eq:bregman_split} is the key structural feature of the
mixed-exponent geometry: it lets the density and value-function blocks be
estimated separately, even though \(\mathcal A_\eps\) couples them. The
nonnegativity constraint \(m\ge0\) acts only on the density, while the bounded
truncation in \(K\) acts on the product norm.

The density block has derivative
\begin{equation}\label{eq:grad_phi_m}
D\Phi_m(m)=|m|^{\bbar-2}m\in L^{\bbar'}(\T).
\end{equation}
For \(m\ge0\), \eqref{eq:grad_phi_m} becomes
\(D\Phi_m(m)=m^{\bbar-1}\).

For the value-function block, the derivative is the element of
\(W^{-1,\gbar'}(\T)\) defined by
\[
\pair{D\Phi_u(u)}{v}_{W^{-1,\gbar'},W^{1,\gbar}}
=
\int_\T |Du|^{\gbar-2}Du\cdot Dv\,dx
+
\int_\T |u|^{\gbar-2}u\,v\,dx
\]
for every \(v\in W^{1,\gbar}(\T)\). In distributional form,
\begin{equation}\label{eq:grad_phi_u}
D\Phi_u(u)
=
-\diver\!\bigl(|Du|^{\gbar-2}Du\bigr)
+
|u|^{\gbar-2}u.
\end{equation}
Together, \eqref{eq:grad_phi_m} and \eqref{eq:grad_phi_u} describe the duality
map \(D\Phi:X\to X^*\). 
The second block has the same \(\gbar\)-Laplacian structure as the
regularizing term in \(\mathcal A_\eps\) (see \eqref{eq:operator}), aligning
the Bregman geometry with the coercive part of the operator.
Moreover, \(\Phi\) is a mirror potential in the sense of
Definition~\ref{def:mirror_potential}. Indeed, it is strictly convex, lower
semicontinuous, and supercoercive on \(X\); for some \(c>0\),
\[
\Phi(m,u)
\ge
c\left(
\|m\|_{L^{\bbar}}^{\bbar}
+
\|u\|_{W^{1,\gbar}}^{\gbar}
\right).
\]
Since \(\bbar,\gbar>1\), this implies supercoercivity. The derivative
formulas \eqref{eq:grad_phi_m} and \eqref{eq:grad_phi_u}, together with
\eqref{eq:Jp_lip_pointwise} and H\"older's inequality, show that
\(\Phi\) is continuously Fréchet differentiable and that \(D\Phi\)
is bounded on bounded subsets of \(X\). 
\begin{remark}[A coordinatewise variant]\label{rem:anisotropic}
The two-dimensional experiments of Section~\ref{sec:numerics} use, in place
of \(\Phi_u\), the coordinatewise potential
\[
\Phi_{u,\mathrm{an}}(u)
=
\frac1{\gbar}\int_\T
\Bigl(\sum_{i=1}^d|\partial_iu|^{\gbar}+|u|^{\gbar}\Bigr)\,dx,
\qquad
J_{\mathrm{an}}(u)
=
-\sum_{i=1}^d\partial_i\bigl(|\partial_iu|^{\gbar-2}\partial_iu\bigr)
+|u|^{\gbar-2}u,
\]
with \(J_{\mathrm{an}}=D\Phi_{u,\mathrm{an}}\), and correspondingly replace
the regularization in \eqref{eq:operator} by \(\eps J_{\mathrm{an}}(u)\);
in this variant, the mirror potential and the regularizer are both replaced accordingly. Since
\(\sum_i|\xi_i|^{\gbar}\) and \(|\xi|^{\gbar}\) are equivalent on \(\RR^d\)
with constants depending on \(d\), applying \eqref{eq:p_lap_mono}
and \eqref{eq:vector_bregman_power}, with \(N=1\), to each coordinate gives the analogues
of Theorem~\ref{thm:Aeps_properties}, Lemma~\ref{lem:op_bound},
Lemma~\ref{lem:bregman_lb}, and Lemma~\ref{lem:norm_comparison}, with
constants depending on \(d\). This is a different regularized problem: the
norm equivalence transfers these estimates but does not identify the two
regularizations. The existence and vanishing-regularization results of
Section~\ref{sec:existence} are imported for the isotropic term; 
the proofs in \cite[Section~3]{ferreiraSolvingMeanFieldGames2025}
carry over because \(\Phi_{u,\mathrm{an}}\) is convex,
\(J_{\mathrm{an}}\) is coercive, hemicontinuous, and bounded on bounded sets,
and \(\pair{J_{\mathrm{an}}(u)}{u}
=\gbar\Phi_{u,\mathrm{an}}(u)\ge0\).
Thus the uniform energy estimate and Minty passage remain valid.
For \(d=1\) the two
potentials coincide.
\end{remark}

\subsection{Coercivity of the mixed-exponent Bregman divergence}

We next record the lower bound connecting \(\mathcal D_\Phi\) to the natural
Banach norms of the problem. By \eqref{eq:bregman_split}, it is enough to
control the density and value-function blocks separately.

\begin{lemma}[Lower bound for the mixed-exponent Bregman divergence]\label{lem:bregman_lb}
Let \(\Phi\) be given by \eqref{eq:phi}. Then there exists
\(\sigma>0\), depending only on \(\bbar\) and \(\gbar\), such that for all
\((m_1,u_1),(m_2,u_2)\in K\),
\begin{equation}\label{eq:bregman_lb}
\mathcal D_\Phi\bigl((m_1,u_1),(m_2,u_2)\bigr)
\ge
\sigma
\left(
\|m_1-m_2\|_{L^{\bbar}}^{\bbar}
+
\|u_1-u_2\|_{W^{1,\gbar}}^{\gbar}
\right).
\end{equation}
\end{lemma}

\begin{proof}
Since \(m_1,m_2\ge0\), the scalar estimate
\eqref{eq:scalar_bregman_power}, with \(r=\bbar\), gives
\[
\mathcal D_{\Phi_m}(m_1,m_2)
\ge
\frac1{\bbar}
\|m_1-m_2\|_{L^{\bbar}}^{\bbar}.
\]
For the value-function block, applying \eqref{eq:vector_bregman_power} with
\(r=\gbar\) to the gradient term and to the zero-order term gives, for a
constant \(c_{\gbar}>0\),
\[
\mathcal D_{\Phi_u}(u_1,u_2)
\ge
c_{\gbar}\|Du_1-Du_2\|_{L^{\gbar}}^{\gbar}
+
c_{\gbar}\|u_1-u_2\|_{L^{\gbar}}^{\gbar}
\geq c \|u_1-u_2\|_{W^{1,\gbar}}^{\gbar}.
\]
Combining the density
and value-function estimates yields \eqref{eq:bregman_lb}.
\end{proof}

For the convergence proof, we also need a single exponent on the bounded set
\(K\).

\begin{lemma}[Exponent comparison on \(K\)]\label{lem:norm_comparison}
Let
\[
\kappa:=\max\{\bbar,\gbar\}.
\]
Let \((m_i,u_i)\in K\), \(i=1,2\). Then there exists \(C_R>0\), depending only
on \(R\), \(\bbar\), and \(\gbar\), such that
\begin{equation}\label{eq:norm_comparison}
\mathcal D_\Phi\bigl((m_1,u_1),(m_2,u_2)\bigr)
\ge
C_R
\left(
\|m_1-m_2\|_{L^{\bbar}}^\kappa
+
\|u_1-u_2\|_{W^{1,\gbar}}^\kappa
\right).
\end{equation}
\end{lemma}

\begin{proof}
Since \(K\subset\overline B_X(0,R)\), the triangle inequality gives
\[
\|m_1-m_2\|_{L^{\bbar}}\le2R,
\qquad
\|u_1-u_2\|_{W^{1,\gbar}}\le2R.
\]
If \(0\le a\le2R\) and \(0<q\le\kappa\), then
\[
a^q\ge (2R)^{q-\kappa}a^\kappa.
\]
Applying this with \(q=\bbar\) and \(q=\gbar\) in
\eqref{eq:bregman_lb} gives \eqref{eq:norm_comparison}.
\end{proof}

\section{Mirror-type algorithms}\label{sec:algorithms}

Throughout this section, \(K\) denotes the bounded admissible set fixed in
\eqref{eq:bounded_K}. For every fixed \(\eps>0\), the regularized solution
\(z_\eps=(m_\eps,u_\eps)\) is the unique solution of
\(\VI(\mathcal A_\eps,K)\), by Lemma~\ref{lem:bounded_admissible_set}. We define
the mirror iteration directly on this set.

We write \(\lambda_n>0\) for the step size at iteration \(n\); the summability
conditions imposed on \((\lambda_n)\) are stated below.

Given \(x_0\in K\), the basic one-step mirror iteration is
\begin{equation}\label{eq:one_step_algorithm}
A_n:=\mathcal A_\eps(x_n),
\qquad
x_{n+1}:=\Mirr_{x_n}(\lambda_n A_n).
\end{equation}
The convergence analysis uses the quantitative monotonicity
\eqref{eq:unif_mono}, together with the operator bound and the Bregman
estimates established above.

Here \(\Mirr_x\) is the constrained mirror map from
Definition~\ref{def:mirror_map}, associated with the potential
\eqref{eq:phi}.

We also consider a frozen two-step variant, given in
Algorithm~\ref{alg:semenov}, in which both mirror steps use
the same operator value \(A_n=\mathcal A_\eps(x_n)\).
\begin{remark}[Doubled-step identity]\label{rem:doubled_step}
For the separable potential~\eqref{eq:phi}, when the ball constraint is
inactive in both mirror steps, the frozen two-step update satisfies
\[
x_{n+1}=\Mirr_{x_n}(2\lambda_n A_n).
\]
Indeed, the density update is a clipped translation in
\(m^{\bbar-1}\), and
\[
[[a-b]_+-b]_+=[a-2b]_+
\qquad\text{for }a\ge0,\ b\in\RR,
\]
where \([s]_+:=\max\{s,0\}\); the two translations in \(D\Phi_u(u)\) add
directly. Thus density clipping preserves the identity, whereas the two
updates can differ when the ball constraint is active. The identity is
specific to this geometry and constraint regime.
\end{remark}

For the convergence theory of both iterations, the step-size sequence
\((\lambda_n)_{n\ge0}\subset(0,\infty)\) is assumed to satisfy
\begin{equation}\label{eq:stepsize_conditions}
\sum_{n=0}^{\infty}\lambda_n=\infty,
\qquad
\sum_{n=0}^{\infty}\lambda_n^{\kappa'}<\infty,
\qquad
\kappa:=\max\{\bbar,\gbar\},
\qquad
\kappa':=\frac{\kappa}{\kappa-1}.
\end{equation}
A concrete family of step-size sequences satisfying
\eqref{eq:stepsize_conditions} is
\[
\lambda_n=(n+1)^{-q}
\]
for any
\[
\frac1{\kappa'}<q\le1.
\]
Since \(1/\kappa'=1-1/\kappa<1\), this interval is nonempty.

\begin{algorithm}[H]
\caption{Frozen two-step mirror method}
\label{alg:semenov}
\begin{algorithmic}[1]
\REQUIRE \(x_0 \in K\)
\STATE Choose a step-size sequence \((\lambda_n)_{n\ge0}\) satisfying \eqref{eq:stepsize_conditions}
\STATE \(A_0 \gets \mathcal A_\eps(x_0)\)
\FOR{\(n = 0,1,2,\ldots\)}
\STATE \(y_n \gets \Mirr_{x_n}(\lambda_n A_n)\)
\STATE \(x_{n+1} \gets \Mirr_{y_n}(\lambda_n A_n)\)
\STATE \(A_{n+1} \gets \mathcal A_\eps(x_{n+1})\)
\ENDFOR
\end{algorithmic}
\end{algorithm}

Both iterations remain in \(K\) and use one new operator evaluation
per outer iteration. The one-step method uses one mirror solve;
Algorithm~\ref{alg:semenov} uses two with the same operator value.
Section~\ref{sec:convergence} analyzes both iterations.
The numerical experiments in Section~\ref{sec:numerics} use the
frozen two-step variant and compare it with the one-step method.

\section{Convergence analysis}\label{sec:convergence}

Throughout this section, the regularization parameter \(\eps>0\) is fixed, and
Assumptions~\ref{ass:base} and~\ref{ass:unifmono} hold.
Let
\[
z_\eps^*:=(m_\eps,u_\eps)=z_\eps\in K
\]
denote the unique solution of \(\VI(\mathcal A_\eps,K)\), as given by
Lemma~\ref{lem:bounded_admissible_set}. The goal is to prove strong convergence
in the \(X\)-norm of the one-step mirror method and the frozen
two-step variant introduced in Section~\ref{sec:algorithms}.

The one-step iteration is given by~\eqref{eq:one_step_algorithm}.
For reference, we write the frozen two-step iteration,
Algorithm~\ref{alg:semenov}, in equation form. Given
\(x_0\in K\), set
\begin{equation}\label{eq:variable_step_algorithm}
A_n:=\mathcal A_\eps(x_n),
\qquad
y_n:=\Mirr_{x_n}(\lambda_n A_n),
\qquad
x_{n+1}:=\Mirr_{y_n}(\lambda_n A_n),
\qquad n\ge0.
\end{equation}
Because both mirror steps are taken over \(K\), all iterates \(x_n\) and
\(y_n\) remain in \(K\).
The proof has three steps. First, the mirror optimality condition gives a
one-step Bregman inequality. Second, the boundedness of \(\mathcal A_\eps\) on \(K\) and the
mirror-increment bound control the remaining pairing by a summable
term of order \(\lambda_n^{\kappa'}\). Finally, the resulting quasi-Fejér
estimate and the quantitative monotonicity of \(\mathcal A_\eps\) identify the
strong limit as \(z_\eps^*\).

\subsection{The one-step Bregman estimate}

We first record the geometric estimate behind the descent argument for the
mirror map on \(K\).

\begin{proposition}[One-step mirror inequality]\label{prop:one_step}
Let \(\Phi\) be a mirror potential, let \(w\in K\),
let \(\xi\in X^*\), and let \(\tau>0\).
Set
\[
z^+:=\Mirr_w(\tau\xi).
\]
Then, for every \(\zeta\in K\),
\begin{equation}\label{eq:one_step_general}
\mathcal D_\Phi(\zeta,z^+)
\le
\mathcal D_\Phi(\zeta,w)
-
\mathcal D_\Phi(z^+,w)
+
\tau\pair{\xi}{\zeta-z^+}.
\end{equation}
\end{proposition}
\begin{proof}
By Lemma~\ref{lem:mirror_step_props}, the optimality condition for
\(z^+=\Mirr_w(\tau\xi)\) gives
\[
\pair{D\Phi(z^+)-D\Phi(w)+\tau\xi}{\zeta-z^+}\ge0
\qquad
\text{for all }\zeta\in K.
\]
Using the three-point identity \eqref{eq:three_point} with
\((a,b,c)=(\zeta,z^+,w)\), we obtain
\[
\mathcal D_\Phi(\zeta,w)
-
\mathcal D_\Phi(\zeta,z^+)
-
\mathcal D_\Phi(z^+,w)
=
\pair{D\Phi(z^+)-D\Phi(w)}{\zeta-z^+}.
\]
Using this identity in the optimality condition gives
\eqref{eq:one_step_general}.

\end{proof}

\subsection{Mirror-increment estimates}

We now collect the estimates needed to control the mirror increments and the
remaining pairing in the descent argument. All constants in this
subsection may depend on \(R\), on the fixed parameter \(\eps\), and on the
structural constants, but not on the iteration index \(n\) or on the particular
points being estimated.

Lemma~\ref{lem:norm_comparison} and the inequality
\[
a^\kappa+b^\kappa\ge 2^{1-\kappa}(a+b)^\kappa
\qquad
\text{for all }a,b\ge0
\]
imply that, after decreasing the constant if necessary, there exists
\(\sigma_R>0\) such that
\begin{equation}\label{eq:bounded_bregman}
\mathcal D_\Phi(z_1,z_2)
\ge
\sigma_R\|z_1-z_2\|_X^\kappa
\qquad
\text{for all }z_1,z_2\in K.
\end{equation}

By Lemma~\ref{lem:op_bound}, setting
\(\Xi_{R,\eps}:=C_R(1+\eps)\), we have, for every \(z\in K\),
\begin{equation}\label{eq:operator_bound_REps}
\|\mathcal A_\eps(z)\|_{X^*}
\le
\Xi_{R,\eps},
\end{equation}
and \(\Xi_{R,\eps}\) is uniformly bounded for \(0<\eps\le1\).

The next estimate controls the size of one mirror increment.

\begin{lemma}[Size of one mirror increment]\label{lem:mirror_increment}
Let \(w\in K\), let \(\xi\in X^*\) satisfy
\[
\|\xi\|_{X^*}\le\Xi_{R,\eps},
\]
and let \(\tau>0\). Let
\[
z^+:=\Mirr_w(\tau\xi).
\]
Then,
\begin{equation}\label{eq:mirror_increment}
\|z^+-w\|_X
\le
\left(
\frac{\tau\Xi_{R,\eps}}{\sigma_R}
\right)^{1/(\kappa-1)}.
\end{equation}
\end{lemma}
\begin{proof}
By the definition of the mirror map, \(z^+\) minimizes
\[
y\mapsto \tau\pair{\xi}{y}+\mathcal D_\Phi(y,w)
\]
over \(K\). Taking \(w\in K\) as a competitor and using
\(\mathcal D_\Phi(w,w)=0\), we obtain
\[
\tau\pair{\xi}{z^+}
+
\mathcal D_\Phi(z^+,w)
\le
\tau\pair{\xi}{w}.
\]
Hence
\[
\mathcal D_\Phi(z^+,w)
\le
\tau\pair{\xi}{w-z^+}
\le
\tau\|\xi\|_{X^*}\|z^+-w\|_X
\le
\tau\Xi_{R,\eps}\|z^+-w\|_X.
\]
Since \(z^+,w\in K\), \eqref{eq:bounded_bregman} gives
\[
\sigma_R\|z^+-w\|_X^\kappa
\le
\mathcal D_\Phi(z^+,w).
\]
If \(z^+=w\), the conclusion is trivial. Otherwise, dividing by
\(\|z^+-w\|_X\) gives
\[
\sigma_R\|z^+-w\|_X^{\kappa-1}
\le
\tau\Xi_{R,\eps},
\]
which proves \eqref{eq:mirror_increment}.
\end{proof}

In both iterations, \(\tau=\lambda_n\) and the direction \(A_n\) is bounded
by \eqref{eq:operator_bound_REps}. Hence Lemma~\ref{lem:mirror_increment}
gives
\[
\|x_{n+1}-x_n\|_X
\le
\Bigl(\frac{\lambda_n\Xi_{R,\eps}}{\sigma_R}\Bigr)^{1/(\kappa-1)}
\]
for the one-step iteration, and the same bound for \(\|y_n-x_n\|_X\) and
\(\|x_{n+1}-y_n\|_X\) in the frozen two-step iteration. Since
\(1+1/(\kappa-1)=\kappa'\), the pairings
\(\lambda_n\pair{A_n}{x_n-x_{n+1}}\) that remain in the descent argument are
bounded by \(\Xi_{R,\eps}\lambda_n\|x_{n+1}-x_n\|_X\le C_{R,\eps}\lambda_n^{\kappa'}\),
which is summable.

\subsection{Descent estimates}

Define the fixed-\(\eps\) error functional
\begin{equation}\label{eq:energy_eps}
\mathcal E_\eps(z)
:=
\|m-m_\eps\|_{L^{\bbar}}^{\bbar}
+
\eps\|u-u_\eps\|_{W^{1,\gbar}}^{\gbar},
\qquad
z=(m,u)\in K.
\end{equation}
By the quantitative monotonicity estimate \eqref{eq:unif_mono}, there exists 
\(c_{\mathrm{mon}}>0\), independent of \(z\) and \(\eps\), such that
\begin{equation}\label{eq:energy_mono} 
\pair{\mathcal A_\eps(z)-\mathcal A_\eps(z_\eps^*)}{z-z_\eps^*} \ge c_{\mathrm{mon}}\,\mathcal E_\eps(z) \qquad \text{for all }z\in K. 
\end{equation}

The next proposition gives discrete Lyapunov inequalities for both
iterations. The mirror moves produce the increment terms, the quantitative
monotonicity of \(\mathcal A_\eps\) at the current iterate gives the
dissipative term \(\lambda_n\mathcal E_\eps(x_n)\), and the only loss is the
pairing of the frozen direction with the increment, which the estimates of
Section~\ref{sec:convergence} bound by the summable quantity
\(C_{R,\eps}\lambda_n^{\kappa'}\).

\begin{proposition}[Descent inequalities for the mirror iterations]\label{prop:two_step_estimate}
There exists \(C_{R,\eps}>0\), which can be chosen uniformly for
\(0<\eps\le1\), such that the following hold for every \(n\ge0\).
For the one-step iteration \eqref{eq:one_step_algorithm},
\begin{equation}\label{eq:one_step_estimate}
\mathcal D_\Phi(z_\eps^*,x_{n+1})
+
\sigma_R\|x_{n+1}-x_n\|_X^\kappa
+
c_{\mathrm{mon}}\lambda_n\mathcal E_\eps(x_n)
\le
\mathcal D_\Phi(z_\eps^*,x_n)
+
C_{R,\eps}\lambda_n^{\kappa'}.
\end{equation}
For the frozen two-step iteration \eqref{eq:variable_step_algorithm},
\begin{align}\label{eq:two_step_estimate}
\mathcal D_\Phi(z_\eps^*,x_{n+1})
&+
\sigma_R\|y_n-x_n\|_X^\kappa
+
\sigma_R\|x_{n+1}-y_n\|_X^\kappa
+
2c_{\mathrm{mon}}\lambda_n\mathcal E_\eps(x_n)
\notag\\
&\le
\mathcal D_\Phi(z_\eps^*,x_n)
+
C_{R,\eps}\lambda_n^{\kappa'}.
\end{align}
\end{proposition}

\begin{proof}
Set \(A_n:=\mathcal A_\eps(x_n)\). Since \(x_n\in K\) and \(z_\eps^*\)
solves \(\VI(\mathcal A_\eps,K)\), we have
\(\pair{\mathcal A_\eps(z_\eps^*)}{x_n-z_\eps^*}\ge0\), so that, by
\eqref{eq:energy_mono},
\begin{equation}\label{eq:dissipation_current}
\pair{A_n}{x_n-z_\eps^*}
=
\pair{\mathcal A_\eps(x_n)-\mathcal A_\eps(z_\eps^*)}{x_n-z_\eps^*}
+
\pair{\mathcal A_\eps(z_\eps^*)}{x_n-z_\eps^*}
\ge
c_{\mathrm{mon}}\mathcal E_\eps(x_n).
\end{equation}

For the one-step iteration, apply \eqref{eq:one_step_general} with
\(\zeta=z_\eps^*\), \(w=x_n\), \(\xi=A_n\), \(\tau=\lambda_n\), and
\(z^+=x_{n+1}\), and write
\[
\pair{A_n}{z_\eps^*-x_{n+1}}
=
-\pair{A_n}{x_n-z_\eps^*}
+
\pair{A_n}{x_n-x_{n+1}}.
\]
The first term is bounded by \eqref{eq:dissipation_current}, the Bregman
increment by \eqref{eq:bounded_bregman}, and, by
\eqref{eq:operator_bound_REps} and Lemma~\ref{lem:mirror_increment},
\[
\lambda_n\bigl|\pair{A_n}{x_n-x_{n+1}}\bigr|
\le
\lambda_n\Xi_{R,\eps}\|x_{n+1}-x_n\|_X
\le
\Xi_{R,\eps}^{\kappa'}\sigma_R^{-1/(\kappa-1)}\lambda_n^{\kappa'},
\]
where we used \(1+1/(\kappa-1)=\kappa'\). This proves
\eqref{eq:one_step_estimate}.

For the frozen two-step iteration, apply \eqref{eq:one_step_general} twice
with the same direction \(\xi=A_n\), first with \(w=x_n\), \(z^+=y_n\), and
then with \(w=y_n\), \(z^+=x_{n+1}\), and add the two inequalities:
\[
\mathcal D_\Phi(z_\eps^*,x_{n+1})
+
\mathcal D_\Phi(y_n,x_n)
+
\mathcal D_\Phi(x_{n+1},y_n)
\le
\mathcal D_\Phi(z_\eps^*,x_n)
+
\lambda_n\bigl(\pair{A_n}{z_\eps^*-y_n}+\pair{A_n}{z_\eps^*-x_{n+1}}\bigr),
\]
where
\[
\pair{A_n}{z_\eps^*-y_n}+\pair{A_n}{z_\eps^*-x_{n+1}}
=
-2\pair{A_n}{x_n-z_\eps^*}
+
\pair{A_n}{x_n-y_n}
+
\pair{A_n}{x_n-x_{n+1}}.
\]
The first term is bounded by \eqref{eq:dissipation_current}. By
Lemma~\ref{lem:mirror_increment},
\(\|x_n-y_n\|_X\le(\lambda_n\Xi_{R,\eps}/\sigma_R)^{1/(\kappa-1)}\) and
\(\|x_n-x_{n+1}\|_X\le\|x_n-y_n\|_X+\|y_n-x_{n+1}\|_X
\le2(\lambda_n\Xi_{R,\eps}/\sigma_R)^{1/(\kappa-1)}\), so the last two
pairings contribute at most
\(3\Xi_{R,\eps}^{\kappa'}\sigma_R^{-1/(\kappa-1)}\lambda_n^{\kappa'}\).
Together with \eqref{eq:bounded_bregman} this proves
\eqref{eq:two_step_estimate}.

In both cases \(C_{R,\eps}=3\Xi_{R,\eps}^{\kappa'}\sigma_R^{-1/(\kappa-1)}\)
works; it depends on \(R\) through \(\Xi_{R,\eps}\) and \(\sigma_R\), and is
bounded by \(C_R(1+\eps)^{\kappa'}\).
\end{proof}

\subsection{Strong convergence for fixed \texorpdfstring{\(\eps>0\)}{epsilon > 0}}


We now pass from the descent estimate to convergence using a quasi-Fej\'er
argument.

\begin{lemma}[Quasi-Fej\'er summability]\label{lem:quasi_fejer}
Let \((E_n)\), \((a_n)\), \((b_n)\), and \((e_n)\) be nonnegative sequences, and
let \((\lambda_n)\subset(0,\infty)\). Suppose that
\begin{equation}\label{eq:quasi_fejer}
E_{n+1}+a_n+\lambda_n b_n
\le
E_n+e_n
\qquad
\text{for all }n\ge0.
\end{equation}
If
\[
\sum_{n=0}^{\infty}e_n<\infty,
\]
then \((E_n)\) converges and
\[
\sum_{n=0}^{\infty}a_n<\infty,
\qquad
\sum_{n=0}^{\infty}\lambda_n b_n<\infty.
\]
If, in addition,
\[
\sum_{n=0}^{\infty}\lambda_n=\infty,
\]
then
\[
\liminf_{n\to\infty} b_n=0.
\]
\end{lemma}

\begin{proof}
Let
\[
s_n:=\sum_{j=n}^{\infty}e_j.
\]
Then \(s_n\ge0\), \(s_n\to0\), and \(s_n-s_{n+1}=e_n\). Define
\[
\widetilde E_n:=E_n+s_n.
\]
Using \eqref{eq:quasi_fejer}, we find
\[
\widetilde E_{n+1}+a_n+\lambda_n b_n
=
E_{n+1}+s_{n+1}+a_n+\lambda_n b_n
\le
E_n+e_n+s_{n+1}
=
E_n+s_n
=
\widetilde E_n.
\]
Thus \((\widetilde E_n)\) is nonincreasing and bounded below, hence convergent.
Summing this estimate from \(n=0\) to \(N\) and using the telescoping of
\((\widetilde E_n)\), we obtain
\[
\sum_{n=0}^{N}a_n
+
\sum_{n=0}^{N}\lambda_n b_n
\le
\widetilde E_0-\widetilde E_{N+1}
\le
\widetilde E_0.
\]
Letting \(N\to\infty\) gives
\[
\sum_{n=0}^{\infty}a_n<\infty,
\qquad
\sum_{n=0}^{\infty}\lambda_n b_n<\infty.
\]
Since \(s_n\to0\), the sequence \((E_n)\) also converges.

If \(\sum_{n=0}^\infty\lambda_n=\infty\) and \(\liminf_{n\to\infty} b_n>0\), then \(\sum_{n=0}^\infty\lambda_n b_n=\infty\), contradicting the summability above. Hence \(\liminf_{n\to\infty} b_n=0\).
\end{proof}

\begin{theorem}[Strong convergence of the mirror iterations]
\label{thm:main_convergence}
Suppose Assumptions~\ref{ass:base} and~\ref{ass:unifmono} hold. Fix
\(\eps>0\), let \(z_\eps^*\) be the unique solution of
\(\VI(\mathcal A_\eps,K)\), let \(x_0\in K\), and let \((x_n)\) be generated either by the one-step method
\eqref{eq:one_step_algorithm} or by the frozen two-step method
\eqref{eq:variable_step_algorithm}. If
the step-size sequence satisfies
\eqref{eq:stepsize_conditions}, then
\[
x_n\to z_\eps^*
\qquad\text{strongly in }X.
\]
For the frozen two-step method, also \(y_n\to z_\eps^*\) strongly in \(X\).
\end{theorem}

\begin{proof}
Apply Proposition~\ref{prop:two_step_estimate} and
Lemma~\ref{lem:quasi_fejer} with
\[
E_n:=\mathcal D_\Phi(z_\eps^*,x_n),
\]
\[
a_n:=
\begin{cases}
\sigma_R\|x_{n+1}-x_n\|_X^\kappa,
&\text{for \eqref{eq:one_step_algorithm}},\\
\sigma_R\bigl(\|y_n-x_n\|_X^\kappa
+\|x_{n+1}-y_n\|_X^\kappa\bigr),
&\text{for \eqref{eq:variable_step_algorithm}}.
\end{cases}
\]
\[
b_n:=c_{\mathrm{mon}}\mathcal E_\eps(x_n),
\qquad
e_n:=C_{R,\eps}\lambda_n^{\kappa'},
\]
where for the frozen two-step iteration the factor \(2\) in
\eqref{eq:two_step_estimate} is discarded.
Since \(\sum_n \lambda_n^{\kappa'}<\infty\), we have
\[
\sum_{n=0}^{\infty} e_n<\infty,
\]
and therefore the sequence
\(\mathcal D_\Phi(z_\eps^*,x_n)\) converges. Since
\(\sigma_R>0\) and \(c_{\mathrm{mon}}>0\), the corresponding summability statements
give
\[
\sum_{n=0}^{\infty}a_n<\infty,
\]
and
\[
\sum_{n=0}^{\infty}
\lambda_n\mathcal E_\eps(x_n)
<\infty.
\]
Because \(\sum_{n=0}^{\infty}\lambda_n=\infty\), Lemma~\ref{lem:quasi_fejer} gives
\[
\liminf_{n\to\infty} b_n=0,
\]
that is, \(\liminf_{n\to\infty}\mathcal E_\eps(x_n)=0\).
Hence there exists a subsequence \((n_j)\) such that
\[
\mathcal E_\eps(x_{n_j})\to0.
\]
Writing \(x_n=(m_n,u_n)\), and using the fact that \(\eps>0\) is fixed, the
definition \eqref{eq:energy_eps} implies
\[
\|m_{n_j}-m_\eps\|_{L^{\bbar}}\to0,
\qquad
\|u_{n_j}-u_\eps\|_{W^{1,\gbar}}\to0.
\]
Thus
\[
x_{n_j}\to z_\eps^*
\qquad
\text{strongly in }X.
\]

Since \(\Phi\) is continuously Fréchet differentiable,
\(z\mapsto\mathcal D_\Phi(z_\eps^*,z)\) is norm-continuous. Therefore,
\[
\mathcal D_\Phi(z_\eps^*,x_{n_j})\to0.
\]
Since the full sequence
\[
\mathcal D_\Phi(z_\eps^*,x_n)
\]
is convergent and has a subsequence converging to zero, its limit must be zero.
Therefore,
\[
\mathcal D_\Phi(z_\eps^*,x_n)\to0.
\]
The global lower Bregman bound \eqref{eq:bregman_lb} yields
\[
\|m_n-m_\eps\|_{L^{\bbar}}\to0,
\qquad
\|u_n-u_\eps\|_{W^{1,\gbar}}\to0.
\]
Thus
\[
x_n\to z_\eps^*
\qquad
\text{strongly in }X.
\]

For the frozen two-step method, the summability of \(a_n\) implies
\[
\|y_n-x_n\|_X\to0.
\]
Therefore
\[
\|y_n-z_\eps^*\|_X
\le
\|y_n-x_n\|_X+\|x_n-z_\eps^*\|_X
\to0,
\]
and so \(y_n\to z_\eps^*\) strongly in \(X\).
\end{proof}

\begin{remark}[Fixed-regularization nature of the convergence theorem]
Theorem~\ref{thm:main_convergence} should be read together with the
vanishing-regularization results in Theorems~\ref{thm:exist_limit}
and~\ref{thm:strong_limit}. For each fixed \(\eps>0\), the mirror iteration
converges strongly in \(X\) to the unique regularized solution \(z_\eps\). 
Along any vanishing sequence \(\eps_k\downarrow0\), subsequences of
\(z_{\eps_k}\) converge to solutions of the unregularized MFG, with strong convergence in \(X\) established in
Theorem~\ref{thm:strong_limit}.
The constants \(\Xi_{R,\eps}\) and \(C_{R,\eps}\)
can be bounded uniformly for \(0<\eps\le1\).
However, recovering the value-function error from \(\mathcal E_\eps\)
uses
\[
\|u-u_\eps\|_{W^{1,\gbar}}^{\gbar}
\le \eps^{-1}\mathcal E_\eps(m,u).
\]
Thus the argument does not establish convergence in \(X\)
uniformly in \(\eps\).
\end{remark}

\begin{remark}[Sequential accuracy and a best-iterate density bound]
Let \(x_n^\eps=(m_n^\eps,u_n^\eps)\) denote the iterates of either
\eqref{eq:one_step_algorithm} or~\eqref{eq:variable_step_algorithm} for fixed
\(\eps>0\), let \(z_\eps\) be the regularized solution, and let \(z\) be a
limit of \(z_{\eps_k}\) as in Theorem~\ref{thm:strong_limit}. Then
\[
\|x_n^\eps-z\|_X
\le
\|x_n^\eps-z_\eps\|_X+\|z_\eps-z\|_X,
\]
so, given \(\delta>0\), one first chooses \(k\) with
\(\|z_{\eps_k}-z\|_X\le\delta/2\) by Theorem~\ref{thm:strong_limit}, and then,
for \(\eps=\eps_k\), \(n\) with \(\|x_n^\eps-z_\eps\|_X\le\delta/2\) by
Theorem~\ref{thm:main_convergence}. Neither choice is quantified by the
present results.

For the density, a quantitative statement is available. Summing
\eqref{eq:one_step_estimate}, or \eqref{eq:two_step_estimate}, from \(n=0\)
to \(N\) and dropping the nonnegative increment terms gives
\[
c_{\mathrm{mon}}
\sum_{n=0}^N \lambda_n \mathcal E_\eps(x_n^\eps)
\le
\mathcal D_\Phi(z_\eps,x_0)
+
C_{R,\eps}\sum_{n=0}^N \lambda_n^{\kappa'},
\]
and hence, since
\(\|m_n^\eps-m_\eps\|_{L^{\bbar}}^{\bbar}\le\mathcal E_\eps(x_n^\eps)\),
\[
\min_{0\le n\le N}
\|m_{n}^\eps-m_\eps\|_{L^{\bbar}}
\le
\left[
\frac{
\mathcal D_\Phi(z_\eps,x_0)
+
C_{R,\eps}\sum_{n=0}^N \lambda_n^{\kappa'}
}{
c_{\mathrm{mon}}\sum_{n=0}^N\lambda_n
}
\right]^{1/\bbar},
\]
which vanishes as \(N\to\infty\) because \(\sum\lambda_n=\infty\) and
\(\sum\lambda_n^{\kappa'}<\infty\). For a common step-size sequence this bound
is uniform for \(0<\eps\le1\): \(C_{R,\eps}\) is uniformly bounded, as is
\(\mathcal D_\Phi(z_\eps,x_0)\), by the bounds on \(\Phi\) and \(D\Phi\) on
\(K\). The condition \(\mathcal E_\eps(x_n^\eps)\le(\delta/2)^{\bbar}\)
involves the unknown solution \(z_\eps\) and is not a computable stopping
criterion, and the bound gives no uniform last-iterate rate. Combined with
Remark~\ref{rem:density_rate} and the triangle inequality, for
\(0<\eps\le1\) and \(S_N:=\sum_{n=0}^N\lambda_n\),
\[
\min_{0\le n\le N}\|m_n^\eps-m\|_{L^{\bbar}}
\le
CS_N^{-1/\bbar}+C\eps^{1/\bbar};
\]
for \(\lambda_n=(n+1)^{-q}\) with \(1-1/\kappa<q<1\) the first term is
\(C(N+1)^{-(1-q)/\bbar}\).
Here \(C\) may also depend on \(R\) and on the common step-size sequence,
but is independent of \(N\) and \(\eps\in(0,1]\).
\end{remark}

\section{Numerical experiments}\label{sec:numerics}

The experiments have three objectives: to validate the discretization
against prescribed exact solutions, to examine convergence and computational
work under grid refinement, and to compare the one-step mirror iteration with
the frozen two-step variant at matched effective step sizes. Throughout,
\[
H_0(x,p)=\tfrac12|p|^2+b(x)\cdot p,
\qquad
g(m)=m^2,
\qquad
\eps=0.01,
\]
so that \(\alpha=2\), \(\beta=2\), and \(\bbar=\gbar=3\). The mirror
potential is the discrete analogue of \eqref{eq:phi} with the gradient
energy taken coordinatewise (Remark~\ref{rem:anisotropic}); the
value-function block is smoothed with \(\delta=10^{-8}\), the density block
is not, and the regularization in the operator is the corresponding smoothed
coordinatewise map. For \(d=1\) the coordinatewise and isotropic energies
coincide.

\subsection{Discretization and experimental protocol}\label{ssec:disc_v2}

\paragraph{Grid and operators.}
We use \(N^d\) periodic nodes \(x_j=jh\), \(h=1/N\), the one-sided
differences \(D_i^{\pm}u_j=\pm(u_{j\pm e_i}-u_j)/h\), and the monotone
numerical Hamiltonian
\begin{equation}\label{eq:Hh_v2}
H_{h,j}(u)
=\frac12\sum_{i=1}^d\Bigl(\max(D_i^-u_j,0)^2+\min(D_i^+u_j,0)^2\Bigr)
+\sum_{i=1}^d\Bigl(b_i^+(x_j)\,D_i^-u_j+b_i^-(x_j)\,D_i^+u_j\Bigr),
\end{equation}
with \(b_i^+=\max\{b_i,0\}\) and \(b_i^-=\min\{b_i,0\}\), 
which combines a monotone upwind approximation of the quadratic term
with upwinding of the drift.
The transport
operator is the adjoint of the Jacobian of the numerical Hamiltonian
\cite{achdouMeanFieldGames2010}: with \(a_i=\max(D_i^-u,0)+b_i^+\) and
\(c_i=\min(D_i^+u,0)+b_i^-\),
\[
L_h(u)v=\sum_i\bigl(a_iD_i^-v+c_iD_i^+v\bigr),
\qquad
L_h(u)^{\!\top}m=-\sum_i\bigl(D_i^+(a_im)+D_i^-(c_im)\bigr).
\]
The assembled nodal system is
\begin{equation}\label{eq:Fh_v2}
F_{1,h}=-u-H_h(u)+m^2+V_h,
\qquad
F_{2,h}=L_h(u)^{\!\top}m+m-f_h+\eps J_h(u),
\end{equation}
with \(f_h\equiv1\) for the stationary problems and the manufactured source
of Section~\ref{ssec:mms_v2} otherwise. The regularization is the nodal
gradient \(J_h=h^{-d}\nabla_E\Phi_{u,N}\) of the smoothed coordinatewise
energy
\[
\Phi_{u,N}(u)=\frac{h^d}{3}\sum_j\Bigl(\sum_i\bigl((D_i^+u_j)^2+\delta^2\bigr)^{3/2}
+\bigl(u_j^2+\delta^2\bigr)^{3/2}\Bigr),
\qquad
J_h(u)=-\sum_iD_i^-\,j_\delta(D_i^+u)+j_\delta(u),
\]
where \(j_\delta(s)=s\sqrt{s^2+\delta^2}\). Thus \(J_h\) is a
finite-difference discretization of the smoothed coordinatewise operator
\(J_{\mathrm{an},\delta}(u)=-\sum_i\partial_ij_\delta(\partial_iu)+j_\delta(u)\),
which for \(\delta=0\) reduces to \(J_{\mathrm{an}}\) of
Remark~\ref{rem:anisotropic} and, when \(d=1\), to the isotropic operator of
\eqref{eq:operator}. The full discrete potential is
\(\Phi_N(m,u):=\frac{h^d}{3}\sum_jm_j^3+\Phi_{u,N}(u)\). The Euclidean
gradients of the weighted energies are \(h^d\) times the nodal maps, so the
update direction paired with them is the Euclidean dual vector
\(A_n=h^dF_h(x_n)\); below, all formulas are written in the equivalent nodal
convention.

\paragraph{Mirror update.}
For a base state \((m,u)\), a nodal direction \((F_1,F_2)\), and a step
\(\tau>0\), the density block of the constrained mirror step is explicit,
\begin{equation}\label{eq:mstep_v2}
m_j^+=\sqrt{\bigl[m_j^2-\tau F_{1,j}\bigr]_+},
\end{equation}
which is the exact Bregman projection onto \(\{m\ge0\}\); no positive floor
is applied. The value-function update is obtained by solving
\(J_h(u^+)=J_h(u)-\tau F_2\) with Newton's method, initialized at the base
state \(u\). Both mirror implementations use the same inner
solver. Writing
\(r=J_h(u^+)-J_h(u)+\tau F_2\), we use an Armijo backtracking line search,
at most \(50\) Newton iterations, and the stopping condition
\[
\|r\|_{\ell_h^{3/2}}
\le 10^{-2}\tau\min\{\|F_2\|_{\ell_h^{3/2}},\mathrm{tol}_{\mathrm{out}}\}.
\]
If this requested tolerance lies below the estimated double-precision
roundoff level, one final undamped Newton correction is applied.

\paragraph{Constraints.}
Only nonnegativity of the density is enforced; neither the mass nor the mean
of \(u\) is prescribed. The bounded truncation defining \(K\) is monitored in
the discrete norm
\(\|(m,u)\|_{X_h}=\|m\|_{\ell_h^3}+\|u\|_{\ell_h^3}+|u|_{1,3,h}\),
\(|u|_{1,3,h}=(h^d\sum_{j,i}|D_i^+u_j|^3)^{1/3}\), with radius
\(R_h=10\). Over all accepted states and intermediate mirror stages of every
run below, the maximum of \(\|\cdot\|_{X_h}\) is \(3.61\) in one dimension
and \(3.88\) in two, so the truncation is inactive in the reported runs.

\paragraph{Protocol.}
All runs start from \(m^0\equiv1\), \(u^0\equiv0\), use
\(\lambda_n=(n+1)^{-3/4}\), \(n\ge0\), and a budget of \(10^5\) outer
updates (\(5\times10^4\) in two dimensions; never reached). The stopping
threshold is \(\mathcal R_N\le10^{-8}\) for the one-dimensional stationary
problem and \(10^{-6}\) otherwise, where
\begin{equation}\label{eq:RN_v2}
\mathcal R_N=\|F_{1,h}\|_{\ell_h^{3/2}}+\|F_{2,h}\|_{\ell_h^{3/2}},
\qquad
\|w\|_{\ell_h^p}=\Bigl(h^d\sum_j|w_j|^p\Bigr)^{1/p},
\end{equation}
is the nodal operator residual. At a strictly positive discrete equilibrium
the variational inequality reduces to \(F_h=0\); all converged densities
reported below are strictly positive and \(F_{1,h}\le m\) holds at
termination, so \(\mathcal R_N\) coincides there with the complementarity
residual \(\|\min\{m,F_{1,h}\}\|_{\ell_h^{3/2}}+\|F_{2,h}\|_{\ell_h^{3/2}}\).
Density clipping in \eqref{eq:mstep_v2} does occur early: with
\(\lambda_0=1\) the first update of the methods with effective step
\(2\lambda_n\) is \(m^1=\sqrt{[-1-2V_h]_+}\), which vanishes on part of the
grid. Computational work is reported as outer updates and mirror solves,
one per update for the one-step iteration and two for the frozen two-step
iteration. For a
manufactured pair \((m^\star,u^\star)\) sampled at the nodes, the errors are
\[
E_m=\|m_h-m^\star\|_{\ell_h^3},
\qquad
E_{\nabla u}=|u_h-u^\star|_{1,3,h},
\qquad
E_u=\|u_h-u^\star\|_{\ell_h^3},
\]
so that \(E_u+E_{\nabla u}\) is the discrete \(W^{1,3}\) error; rates are
\(\log_2(E(N)/E(2N))\) computed from unrounded values. All computations were performed in Wolfram Language using double-precision arithmetic.

\subsection{Manufactured-solution validation}\label{ssec:mms_v2}

With \(a=2\pi\) and \(x\in\Tt^d\), we prescribe
\[
u^\star=1+\eta_u\sum_{i=1}^d\sin(ax_i),
\qquad
m^\star=1+\rho\sum_{i=1}^d\cos(ax_i),
\qquad
b_i=0.3\cos(ax_i),
\]
with \((\eta_u,\rho)=(0.1,0.15)\) for \(d=1\) and \((0.1,0.1)\) for
\(d=2\). The Hamilton--Jacobi equation determines
\(V^\star=u^\star+\frac12|Du^\star|^2+b\cdot Du^\star-(m^\star)^2\), and the
transport equation determines the source
\[
f^\star=m^\star-\diver\bigl(m^\star(Du^\star+b)\bigr)
+\eps\Bigl(j_\delta(u^\star)-\sum_ij_\delta'(\partial_iu^\star)\,\partial_{ii}u^\star\Bigr),
\]
built with the coordinatewise regularization implemented in
\eqref{eq:Fh_v2}. Both data are sampled at the nodes. The runs use the
one-step iteration with step \(2\lambda_n\) and the threshold \(10^{-6}\).

Table~\ref{tab:mms_v2} and Figure~\ref{fig:mms_v2} report the errors: all
three quantities converge at first order in \(h\) in one and two dimensions,
consistent with the monotone upwind and adjoint transport discretizations.
Tightening the outer threshold to \(10^{-8}\) changes \(E_m\) and
\(E_{\nabla u}\) by less than \(4\times10^{-5}\) (relative) at \(N=512\) and
by less than \(5\times10^{-6}\) on the \(64\times64\) grid, consistent with spatial discretization dominating these two errors. The sampled exact pair
carries a localized
\(O(1)\) transport-consistency defect at the nodes where the drift changes
sign. These nodes form codimension-one grid sets with total discrete weight
\(O(h)\), so the defect contributes \(O(h^{2/3})\) to the
\(\ell_h^{3/2}\) residual of the exact pair; it is distinct from the solution
errors reported here.

\begin{table}[ht]
\centering
\caption{Manufactured solutions: errors and observed rates (\(d=1\), top;
  \(d=2\) with coordinatewise regularization, bottom).}
\label{tab:mms_v2}
\begin{tabular}{lcccccc}
\hline
Grid & \(E_m\) & rate & \(E_{\nabla u}\) & rate & \(E_u\) & rate\\
\hline
\(N=64\)  & \(6.521\times10^{-3}\) & ---  & \(9.433\times10^{-3}\) & ---  & \(7.399\times10^{-4}\) & --- \\
\(N=128\) & \(3.254\times10^{-3}\) & 1.00 & \(4.727\times10^{-3}\) & 1.00 & \(3.775\times10^{-4}\) & 0.97\\
\(N=256\) & \(1.626\times10^{-3}\) & 1.00 & \(2.366\times10^{-3}\) & 1.00 & \(1.911\times10^{-4}\) & 0.98\\
\(N=512\) & \(8.124\times10^{-4}\) & 1.00 & \(1.184\times10^{-3}\) & 1.00 & \(9.603\times10^{-5}\) & 0.99\\
\hline
\(16\times16\) & \(3.685\times10^{-2}\) & ---  & \(4.876\times10^{-2}\) & ---  & \(4.633\times10^{-3}\) & --- \\
\(32\times32\) & \(1.804\times10^{-2}\) & 1.03 & \(2.436\times10^{-2}\) & 1.00 & \(2.316\times10^{-3}\) & 1.00\\
\(64\times64\) & \(8.948\times10^{-3}\) & 1.01 & \(1.220\times10^{-2}\) & 1.00 & \(1.215\times10^{-3}\) & 0.93\\
\hline
\end{tabular}
\end{table}

\begin{figure}[ht]
  \centering
  \begin{minipage}{0.49\textwidth}\centering
    \includegraphics[width=\linewidth]{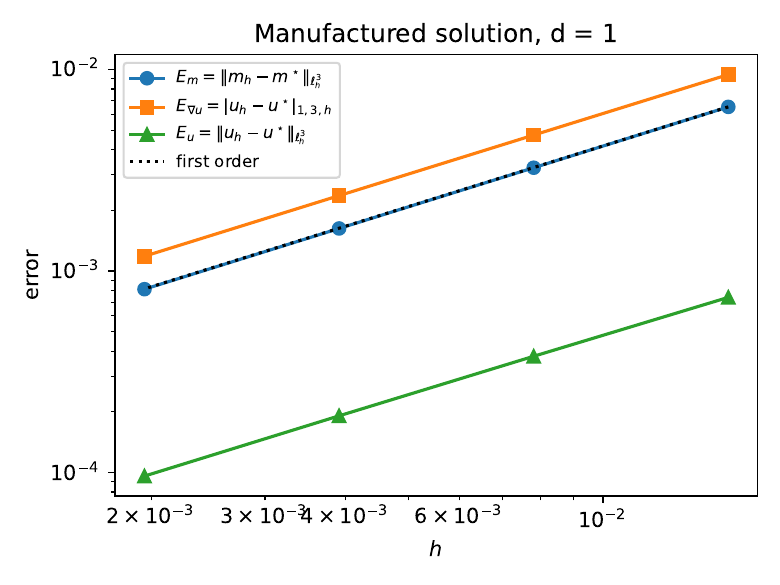}\\
    {\small (a) \(d=1\)}
  \end{minipage}\hfill
  \begin{minipage}{0.49\textwidth}\centering
    \includegraphics[width=\linewidth]{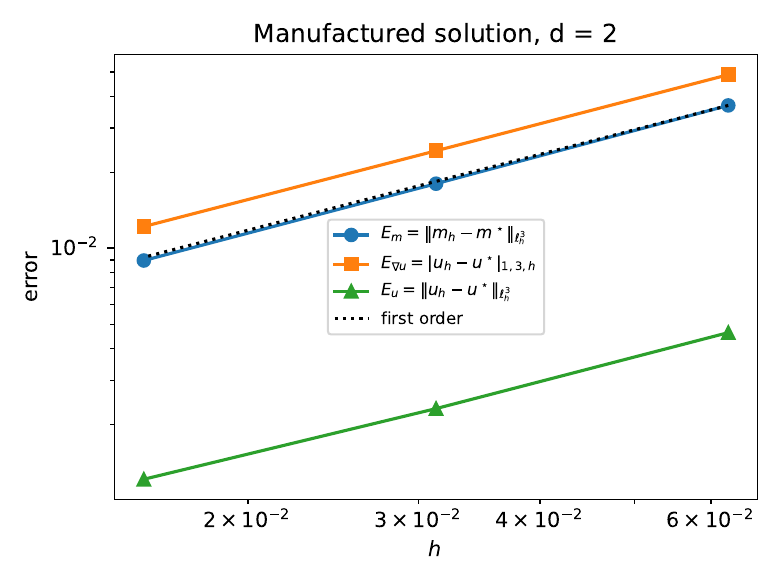}\\
    {\small (b) \(d=2\)}
  \end{minipage}
  \caption{Manufactured solutions: errors versus \(h\) with a first-order
    reference.}
  \label{fig:mms_v2}
\end{figure}

\subsection{Stationary problems and grid refinement}\label{ssec:stationary_v2}

\paragraph{One dimension.}
We solve the stationary problem on \(\Tt\) with \(b(x)=\cos(2\pi x)\),
\(V(x)=\sin(2\pi x)\), \(f\equiv1\), and \(N=64,\dots,1024\), using the
frozen two-step iteration and the threshold \(10^{-8}\)
(Table~\ref{tab:refine_v2}). Over the tested grids the outer update count
varies by less than \(15\%\), and the residual histories in
Figure~\ref{fig:conv_v2}(a) nearly coincide: the residual decreases in a
staircase pattern whose plateaus lengthen with \(N\), the transport
component dominating during the plateaus. The converged densities are
positive, with range \([0.488,1.589]\) at \(N=64\) and \([0.483,1.586]\) at
\(N=1024\), and the discrete balance identity
\(h\sum_jm_j+\eps h\sum_jj_\delta(u_j)=h\sum_jf_{h,j}\), the analogue of
testing the regularized transport equation with \(v\equiv1\), holds at
termination to \(2\times10^{-16}\). 
The mean number of Newton steps per mirror solve increases
from \(1.97\) at \(N=64\) to \(4.58\) at \(N=1024\).

\paragraph{Two dimensions.}
On \(\Tt^2\) we take \(b(x)=(\cos2\pi x_1,\cos2\pi x_2)\),
\(V(x)=\sin2\pi x_1+\sin2\pi x_2\), \(f\equiv1\), the threshold
\(10^{-6}\), and the frozen two-step iteration. The runs require
\(18{,}000\)--\(20{,}000\) outer updates, with variation below \(10\%\)
across the three grids (Table~\ref{tab:refine_v2},
Figure~\ref{fig:conv_v2}(b)). The converged density is positive but
develops a sharper low-density region under refinement, its minimum
decreasing from \(0.029\) to \(0.004\); early iterates have zero-density
nodes on all grids. The balance identity holds to \(3\times10^{-16}\).
Figure~\ref{fig:2d_density} shows the equilibrium on the \(64\times64\)
grid.

\begin{table}[ht]
\centering
\caption{Frozen two-step iteration on the stationary problems: outer
  updates to the threshold (\(10^{-8}\) for \(d=1\), \(10^{-6}\) for
  \(d=2\)), mirror solves, and minimum of the converged density.}
\label{tab:refine_v2}
\begin{tabular}{lcrrc}
\hline
Grid & \(h\) & outer updates & mirror solves & \(\min_jm_j\)\\
\hline
\(N=64\)   & \(1.56\times10^{-2}\) & 5{,}311 & 10{,}622 & 0.488\\
\(N=128\)  & \(7.81\times10^{-3}\) & 5{,}002 & 10{,}004 & 0.487\\
\(N=256\)  & \(3.91\times10^{-3}\) & 4{,}814 &  9{,}628 & 0.485\\
\(N=512\)  & \(1.95\times10^{-3}\) & 4{,}698 &  9{,}396 & 0.484\\
\(N=1024\) & \(9.77\times10^{-4}\) & 4{,}693 &  9{,}386 & 0.483\\
\hline
\(16\times16\) & \(6.25\times10^{-2}\) & 19{,}394 & 38{,}788 & 0.029\\
\(32\times32\) & \(3.13\times10^{-2}\) & 18{,}211 & 36{,}422 & 0.012\\
\(64\times64\) & \(1.56\times10^{-2}\) & 19{,}934 & 39{,}868 & 0.004\\
\hline
\end{tabular}
\end{table}

\begin{figure}[ht]
  \centering
  \begin{minipage}{0.49\textwidth}\centering
    \includegraphics[width=\linewidth]{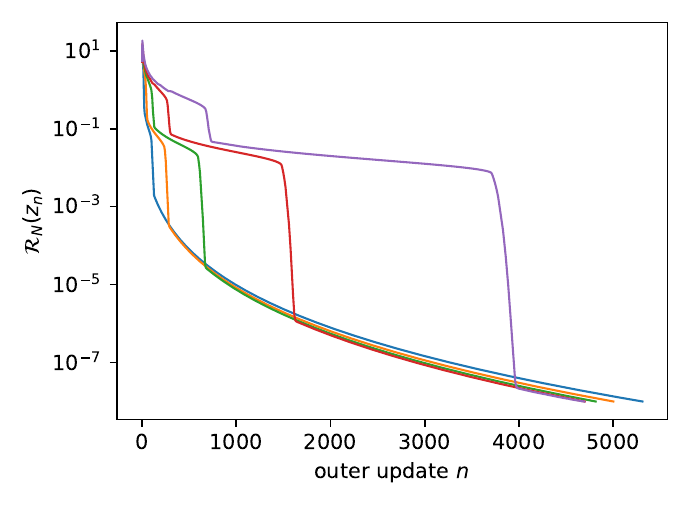}\\
    \includegraphics[width=0.95\linewidth]{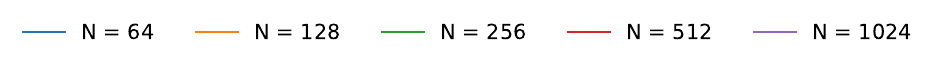}\\
    {\small (a) \(d=1\)}
  \end{minipage}\hfill
  \begin{minipage}{0.49\textwidth}\centering
    \includegraphics[width=\linewidth]{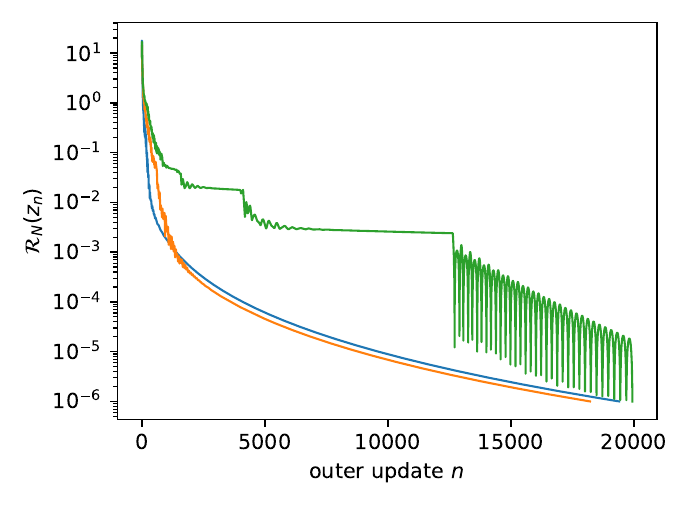}\\
    \includegraphics[width=0.6\linewidth]{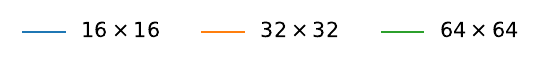}\\
    {\small (b) \(d=2\)}
  \end{minipage}
  \caption{Residual \(\mathcal R_N(x_n)\) versus outer update for the frozen
    two-step iteration under grid refinement (semilog).}
  \label{fig:conv_v2}
\end{figure}

\begin{figure}[ht]
  \centering
  \begin{minipage}{0.49\textwidth}\centering
    \includegraphics[width=\linewidth]{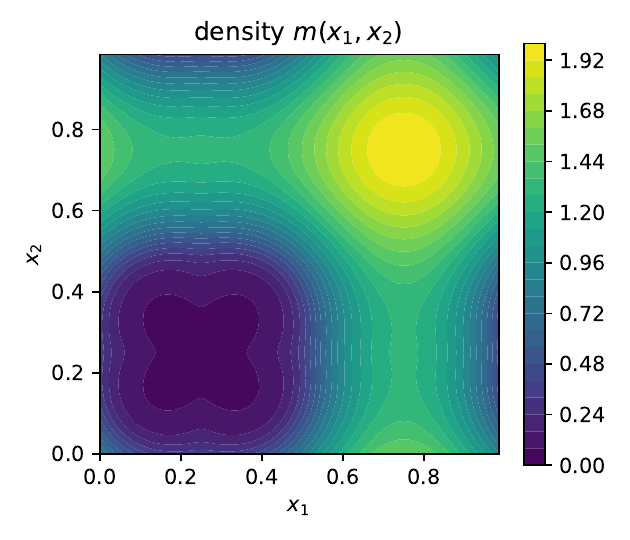}\\
    {\small (a) density $m(x_1,x_2)$}
  \end{minipage}\hfill
  \begin{minipage}{0.49\textwidth}\centering
    \includegraphics[width=\linewidth]{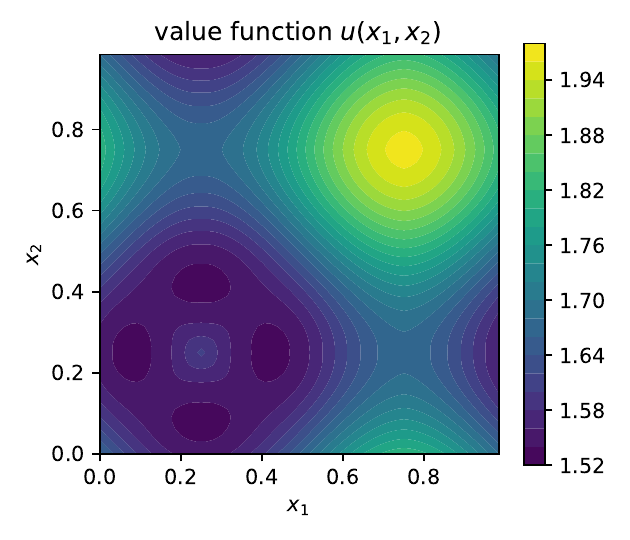}\\
    {\small (b) value function $u(x_1,x_2)$}
  \end{minipage}
  \caption{Two-dimensional equilibrium on the $64\times 64$ grid: density
    (range $[0.004, 2.00]$) and value function.}
  \label{fig:2d_density}
\end{figure}

\subsection{One-step and frozen two-step implementations}\label{ssec:compare_v2}

On the one-dimensional problem we compare the one-step iteration with step
\(\lambda_n\), the one-step iteration with step \(2\lambda_n\), and the
frozen two-step iteration with two moves of \(\lambda_n\). By
Remark~\ref{rem:doubled_step}, valid while the ball constraint is inactive,
the last two are the same exact update also when density clipping occurs;
from a common base state and direction, the two updates agree at \(N=128\)
up to inner-solve accuracy.
Table~\ref{tab:compare_v2} and Figure~\ref{fig:compare_v2} report the
independently evolved runs. At matched effective step sizes, the two
implementations reach the stopping threshold after the same number of outer
updates. The one-step implementation requires one
mirror solve per update, whereas the frozen two-step implementation
requires two. The one-step iteration with the nominal step \(\lambda_n\)
needs about ten times more outer updates: the apparent advantage of the
two-step iteration over this baseline is the effective step size, not the
second mirror solve.

\begin{table}[ht]
\centering
\caption{Three implementations on \(\Tt\): outer updates to
  \(\mathcal R_N\le10^{-8}\) and mirror solves.}
\label{tab:compare_v2}
\begin{tabular}{rlrr}
\hline
\(N\) & implementation & outer updates & mirror solves\\
\hline
64  & one-step, \(\lambda_n\)        & 61{,}981 & 61{,}981\\
64  & one-step, \(2\lambda_n\)       &  5{,}311 &  5{,}311\\
64  & frozen two-step, \(\lambda_n\) &  5{,}311 & 10{,}622\\
\hline
128 & one-step, \(\lambda_n\)        & 57{,}668 & 57{,}668\\
128 & one-step, \(2\lambda_n\)       &  5{,}002 &  5{,}002\\
128 & frozen two-step, \(\lambda_n\) &  5{,}002 & 10{,}004\\
\hline
256 & one-step, \(\lambda_n\)        & 55{,}392 & 55{,}392\\
256 & one-step, \(2\lambda_n\)       &  4{,}814 &  4{,}814\\
256 & frozen two-step, \(\lambda_n\) &  4{,}814 &  9{,}628\\
\hline
\end{tabular}
\end{table}

\begin{figure}[ht]
  \centering
  \includegraphics[width=0.78\textwidth]{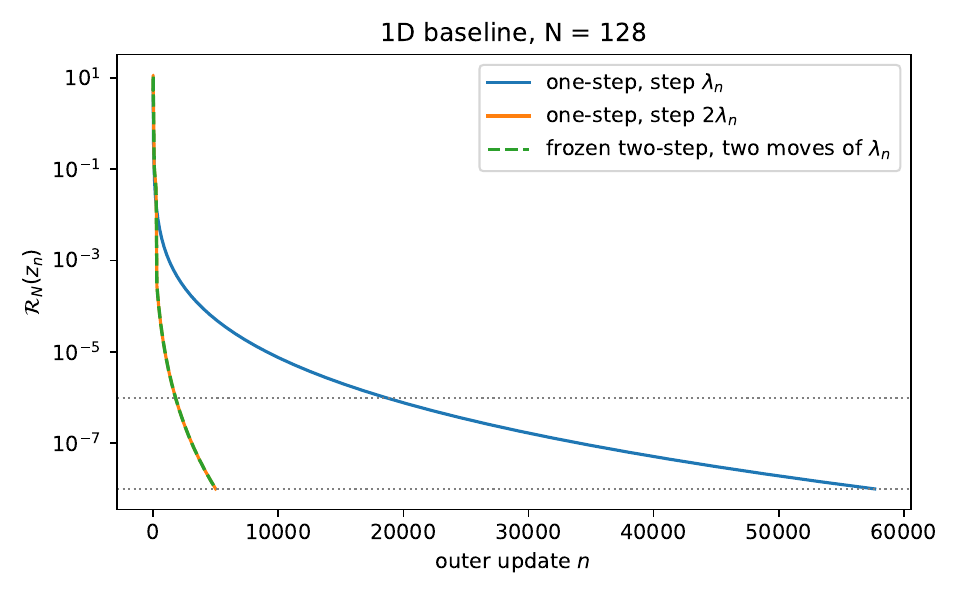}
  \caption{Residual \(\mathcal R_N(x_n)\) versus outer update for the three
    implementations at \(N=128\); the two curves with effective step
    \(2\lambda_n\) coincide at plotting resolution.}
  \label{fig:compare_v2}
\end{figure}

\medskip
In summary, the discretization is first-order accurate on the manufactured
solutions in one and two dimensions, the frozen two-step iteration converges
with nearly grid-independent update counts over the tested ranges, and the
frozen two-step and doubled-step one-step implementations are the same
update to inner-solve accuracy. These are fixed-\(\eps\) observations on
smoothed, inexactly solved mirror steps; the convergence theorem of
Section~\ref{sec:convergence} concerns the exact steps.

\section{Conclusion and future directions}\label{sec:conclusion}

We developed Bregman-projected mirror iterations for a regularized
stationary MFG variational inequality in the natural Banach space
\(L^{\bbar}(\T)\times W^{1,\gbar}(\T)\). The analysis relies on three
ingredients: the low-order \(\gbar\)-Laplacian regularization, the quantitative
monotonicity estimate for \(\mathcal A_\eps\), and the boundedness of
\(\mathcal A_\eps\) on the bounded admissible set \(K\). Together, these
yield strong convergence of the exact one-step iteration and of its frozen
two-step variant for each fixed \(\eps>0\), and quantitative monotonicity
gives strong convergence in \(X\) along weakly convergent
vanishing-regularization sequences. The numerical experiments validate the
discretization against manufactured solutions at first order, show weak
dependence of the outer update counts on the tested grids, and confirm that,
when the ball constraint is inactive, the second frozen mirror solve is
redundant: the frozen two-step and the doubled-step one-step implementations
produce the same update up to inner-solve accuracy.

The scope of the convergence theory is narrower than the full Banach-space
existence theory of \cite{ferreiraSolvingMeanFieldGames2025}. That theory
covers, among other cases, power-growth Hamiltonians with \(\alpha>1\),
\(\beta>0\), nonseparable Hamiltonians, congestion Hamiltonians, and certain
weak-growth Hamiltonians. The present analysis is formulated in the separable power-growth regime
with \(\alpha\ge2\) and \(\beta\ge1\), for which the quantitative
monotonicity and mixed-exponent Bregman estimates used above are available.
Extensions beyond this regime require appropriate coercivity and Bregman
bounds, together with boundedness of the operator on the admissible set.

Several extensions remain natural. 
For nonseparable Hamiltonians \(H(x,p,m)\), one would need quantitative
monotonicity beyond the qualitative Lasry--Lions condition, as well as suitable
regularity in the \(m\)-variable. For congestion-dependent or minimal-growth
Hamiltonians, the underlying function space or the domain of the operator must
be modified before the mirror framework can be applied.

Two further questions are especially relevant for computation. First, the
convergence theorem is a fixed-\(\eps\) result. The operator bound, the
descent constant, and the best-iterate density bound are uniform for
\(0<\eps\le1\); the obstruction to a uniform statement in \(X\) is the factor
\(\eps^{-1}\) incurred when the value-function error is recovered from the
dissipative functional \(\mathcal E_\eps\). Second, the experiments use
smoothed, inexactly solved mirror steps, one initial operator evaluation, and
one evaluation per update. Natural directions are the analysis of inexact
mirror solves with precision-aware stopping, estimates relating discrete
residuals to errors, and the joint control of discretization, iteration, and
regularization errors. Genuine extragradient methods, which evaluate the
operator at the intermediate point, may extend the approach to weaker
monotonicity assumptions.

\bibliographystyle{abbrv}
\bibliography{mfgv9_nn,mybibliography}

@article{achdouMeanFieldGames2010,
  title = {Mean Field Games: Numerical Methods},
  author = {Achdou, Y. and {Capuzzo-Dolcetta}, I.},
  year = 2010,
  journal = {Siam Journal On Numerical Analysis},
  volume = {48},
  number = {3},
  pages = {1136--1162},
  issn = {0036-1429},
  doi = {10.1137/090758477},
  fjournal = {SIAM Journal on Numerical Analysis},
  mrclass = {91A23 (65N06)},
  mrnumber = {2679575 (2011j:91042)},
  mrreviewer = {Diogo Lu\'\is Aguiar Gomes}
}

@article{almullaTwoNumericalApproaches2017,
  title = {Two {{Numerical Approaches}} to {{Stationary Mean-Field Games}}},
  author = {Almulla, Noha and Ferreira, Rita and Gomes, Diogo},
  year = 2017,
  month = dec,
  journal = {Dynamic Games and Applications},
  volume = {7},
  number = {4},
  pages = {657--682},
  issn = {2153-0793},
  doi = {10.1007/s13235-016-0203-5},
  urldate = {2025-01-28},
  langid = {english}
}

@article{Caines2,
  title = {Large-Population Cost-Coupled {LQG} Problems with Nonuniform Agents: Individual-Mass Behavior and Decentralized {$\varepsilon$}-{Nash} Equilibria},
  shorttitle = {Large-Population Cost-Coupled {LQG} Problems with Nonuniform Agents},
  author = {Huang, Minyi and Caines, Peter E. and Malham{\'e}, Roland P.},
  year = 2007,
  month = sep,
  journal = {IEEE Transactions on Automatic Control},
  volume = {52},
  number = {9},
  pages = {1560--1571},
  issn = {0018-9286},
  doi = {10.1109/TAC.2007.904450},
  urldate = {2025-10-17},
  copyright = {https://ieeexplore.ieee.org/Xplorehelp/downloads/license-information/IEEE.html}
}

@article{FeGoTa21,
  title = {Existence of Weak Solutions to Time-Dependent Mean-Field Games},
  author = {Ferreira, R. and Gomes, D. and Tada, T.},
  year = 2021,
  journal = {Nonlinear Analysis},
  volume = {212},
  pages = {Paper No. 112470, 31},
  issn = {0362-546X},
  doi = {10.1016/j.na.2021.112470},
  fjournal = {Nonlinear Analysis. Theory, Methods \& Applications. An International Multidisciplinary Journal},
  mrclass = {49N80 (35A01 35J56 91A16)},
  mrnumber = {4281340}
}

@article{ferreiraSolvingMeanFieldGames2025,
  title = {Solving {{Mean-Field Games}} with {{Monotonicity Methods}} in {{Banach Spaces}}},
  author = {Ferreira, Rita and Gomes, Diogo and {\"U}{\c{c}}er, Melih},
  year = 2026,
  month = aug,
  journal = {SIAM Journal on Mathematical Analysis},
  volume = {58},
  number = {4},
  pages = {4311--4343},
  issn = {1095-7154},
  doi = {10.1137/25M1784521},
  url = {https://doi.org/10.1137/25M1784521}
}

@article{FG2,
  title = {Existence of Weak Solutions to Stationary Mean-Field Games through Variational Inequalities},
  author = {Ferreira, R. and Gomes, D.},
  year = 2018,
  journal = {Siam Journal On Mathematical Analysis},
  volume = {50},
  number = {6},
  pages = {5969--6006},
  issn = {0036-1410},
  doi = {10.1137/16M1106705},
  fjournal = {SIAM Journal on Mathematical Analysis},
  mrclass = {91A13 (35D30 35J88 49J40)},
  mrnumber = {3882950}
}

@article{FGT1,
  title = {Existence of Weak Solutions to First-Order Stationary Mean-Field Games with {{Dirichlet}} Conditions},
  author = {Ferreira, R. and Gomes, D. and Tada, T.},
  year = 2019,
  journal = {Proceedings of the American Mathematical Society},
  volume = {147},
  number = {11},
  pages = {4713--4731},
  issn = {0002-9939,1088-6826},
  doi = {10.1090/proc/14475},
  fjournal = {Proceedings of the American Mathematical Society},
  furthernotes = {This paper makes important theoretical contributions to the study of mean-field games (MFGs) by addressing the challenge of incorporating Dirichlet boundary conditions in first-order stationary monotone MFGs. The authors develop a novel approach to construct solutions by first introducing a regularized problem and proving the existence of a unique weak solution using Schaefer's fixed-point theorem and the monotonicity properties of the MFG. They then employ Minty's method to establish the existence of weak solutions to the original MFG problem. The paper provides a comprehensive mathematical framework, including detailed assumptions, definitions, and proofs of the main theorems. It also discusses the significance of boundary conditions in MFGs and potential applications of the results. The study fills a gap in the literature, as most existing works on first-order MFGs focus on periodic boundary conditions. The paper's findings can serve as a foundation for further research on MFGs with various boundary conditions and their applications in modeling real-world phenomena.},
  mrclass = {35J56 (35A01)},
  mrnumber = {4011507},
  mrreviewer = {AshokAryal}
}

@article{huangLargePopulationStochastic2006,
  title = {Large Population Stochastic Dynamic Games: Closed-Loop {{McKean-Vlasov}} Systems and the {{Nash}} Certainty Equivalence Principle},
  author = {Huang, M. and Malham{\'e}, R. P. and Caines, P. E.},
  year = 2006,
  journal = {Communications in Information and Systems},
  volume = {6},
  number = {3},
  pages = {221--251},
  issn = {1526-7555},
  fjournal = {Communications in Information and Systems},
  mrclass = {91A15 (49L20 91A23)},
  mrnumber = {2346927 (2009f:91008)}
}

@article{lasryMeanFieldGames2007,
  title = {Mean Field Games},
  author = {Lasry, Jean-Michel and Lions, Pierre-Louis},
  year = 2007,
  month = mar,
  journal = {Japanese Journal of Mathematics},
  volume = {2},
  number = {1},
  pages = {229--260},
  issn = {0289-2316, 1861-3624},
  doi = {10.1007/s11537-007-0657-8},
  urldate = {2025-10-17},
  copyright = {http://www.springer.com/tdm},
  langid = {english}
}

@article{ll1,
  title = {Jeux \`a Champ Moyen. {{I}}. {{Le}} Cas Stationnaire},
  author = {Lasry, J.-M. and Lions, P.-L.},
  year = 2006,
  journal = {Comptes Rendus Mathematique. Academie des Sciences. Paris},
  volume = {343},
  number = {9},
  pages = {619--625},
  issn = {1631-073X},
  doi = {10.1016/j.crma.2006.09.019},
  fjournal = {Comptes Rendus Math\'ematique. Acad\'emie des Sciences. Paris},
  mrclass = {91A15 (35J55 35J60 91A06 91A10 91A13)},
  mrnumber = {MR2269875 (2007m:91021)},
  mrreviewer = {Sa\"\id Hamadene}
}

@article{ll2,
  title = {{Jeux \`a champ moyen. II -- Horizon fini et contr\^ole optimal}},
  author = {Lasry, Jean-Michel and Lions, Pierre-Louis},
  year = 2006,
  journal = {Comptes Rendus. Math\'ematique},
  volume = {343},
  number = {10},
  pages = {679--684},
  issn = {1778-3569},
  doi = {10.1016/j.crma.2006.09.018},
  urldate = {2024-08-23},
  langid = {french}
}

@misc{nurbekyanMonotoneInclusionMethods2024,
  title = {Monotone Inclusion Methods for a Class of Second-Order Non-Potential Mean-Field Games},
  author = {Nurbekyan, Levon and Liu, Siting and Chow, Yat Tin},
  year = 2024,
  month = mar,
  number = {arXiv:2403.20290},
  eprint = {2403.20290},
  primaryclass = {cs, math},
  publisher = {arXiv},
  doi = {10.48550/arXiv.2403.20290},
  urldate = {2024-08-29},
  archiveprefix = {arXiv}
}

@misc{nurbekyanNoteConvergenceMonotone2023,
  title = {A Note on the Convergence of the Monotone Inclusion Version of the Primal-Dual Hybrid Gradient Algorithm},
  author = {Nurbekyan, Levon},
  year = 2023,
  month = nov,
  number = {arXiv:2311.03689},
  eprint = {2311.03689},
  primaryclass = {cs, math},
  publisher = {arXiv},
  doi = {10.48550/arXiv.2311.03689},
  urldate = {2024-08-29},
  archiveprefix = {arXiv}
}

@article{W,
  title = {Action Minimizing Stochastic Invariant Measures for a Class of {{Lagrangian}} Systems},
  author = {Wang, K.},
  year = 2008,
  journal = {Communications on Pure and Applied Analysis},
  volume = {7},
  number = {5},
  pages = {1211--1223},
  issn = {1534-0392},
  doi = {10.3934/cpaa.2008.7.1211},
  fjournal = {Communications on Pure and Applied Analysis},
  mrclass = {37H99 (34F05 37J50 60H10 93E03)},
  mrnumber = {2410876 (2010a:37105)}
}

@misc{wuPopulationawareOnlineMirror2024,
  title = {Population-Aware {{Online Mirror Descent}} for {{Mean-Field Games}} by {{Deep Reinforcement Learning}}},
  author = {Wu, Zida and Lauriere, Mathieu and Chua, Samuel Jia Cong and Geist, Matthieu and Pietquin, Olivier and Mehta, Ankur},
  year = 2024,
  month = mar,
  number = {arXiv:2403.03552},
  eprint = {2403.03552},
  primaryclass = {cs, eess},
  publisher = {arXiv},
  doi = {10.48550/arXiv.2403.03552},
  urldate = {2024-08-28},
  archiveprefix = {arXiv}
}

@inproceedings{AKL22,
  author    = {Aubin-Frankowski, Pierre-Cyril and Korba, Anna and L{\'e}ger, Flavien},
  title     = {Mirror descent with relative smoothness in measure spaces, with application to {S}inkhorn and {EM}},
  booktitle = {Advances in Neural Information Processing Systems},
  year      = {2022},
  note      = {\href{https://arxiv.org/abs/2206.08873}{arXiv:2206.08873}}
}

@article{BADLS23,
  author  = {Brice{\~n}o-Arias, Luis M. and Deride, Julio and L{\'o}pez-Rivera, Sergio and Silva, Francisco J.},
  title   = {A primal-dual partial inverse algorithm for constrained monotone inclusions: {A}pplications to stochastic programming and mean field games},
  journal = {Applied Mathematics \& Optimization},
  volume  = {87},
  pages   = {21},
  year    = {2023},
  doi     = {10.1007/s00245-022-09921-9}
}

@article{BT03,
  author  = {Beck, Amir and Teboulle, Marc},
  title   = {Mirror descent and nonlinear projected subgradient methods for convex optimization},
  journal = {Operations Research Letters},
  volume  = {31},
  pages   = {167--175},
  year    = {2003},
  doi     = {10.1016/S0167-6377(02)00231-6}
}

@unpublished{GG26,
  author = {Gevorgyan, Yeva and Gomes, Diogo},
  title  = {A {Riesz-Map Preconditioned Extragradient Method} for {Price Formation} in {Mean Field Games}},
  year   = {2026},
  note   = {\href{https://arxiv.org/abs/2610.13934}{arXiv:2610.13934}},
  eprint = {2610.13934},
  archiveprefix = {arXiv},
  doi = {10.48550/arXiv.2610.13934}
}

@article{HR22,
  author  = {Hieu, Dang Van and Reich, Simeon},
  title   = {Two {B}regman projection methods for solving variational inequalities},
  journal = {Optimization},
  volume  = {71},
  number  = {7},
  pages   = {1777--1802},
  year    = {2022},
  doi     = {10.1080/02331934.2020.1836634}
}

@article{IRS23,
  author  = {Izuchukwu, Chinedu and Reich, Simeon and Shehu, Yekini},
  title   = {One-step {B}regman projection methods for solving variational inequalities in reflexive {B}anach spaces},
  journal = {Optimization},
  volume  = {73},
  number  = {5},
  pages   = {1519--1549},
  year    = {2023},
  doi     = {10.1080/02331934.2022.2150706}
}

@unpublished{LaurMD23,
  author = {Lauri{\`e}re, Mathieu and Musik, Luca and Perrin, Alo{\"\i}s},
  title  = {A mirror descent approach for mean field control},
  year   = {2023},
  note   = {\href{https://hal.science/hal-03972660}{hal-03972660}}
}

@book{NY83,
  author    = {Nemirovsky, Arkadi S. and Yudin, David B.},
  title     = {Problem Complexity and Method Efficiency in Optimization},
  publisher = {John Wiley \& Sons},
  year      = {1983}
}

@inproceedings{PPELP21,
  author    = {P{\'e}rolat, Julien and Perrin, Sarah and Elie, Romuald and Lauri{\`e}re, Mathieu and Piliouras, Georgios and Geist, Matthieu},
  title     = {Scaling up mean field games with online mirror descent},
  booktitle = {Proceedings of the 21st International Conference on Autonomous Agents and Multiagent Systems (AAMAS)},
  year      = {2022},
  note      = {\href{https://arxiv.org/abs/2103.00623}{arXiv:2103.00623}}
}

@article{Sem17,
  author  = {Semenov, Vladimir V.},
  title   = {A version of the mirror descent method to solve variational inequalities},
  journal = {Cybernetics and Systems Analysis},
  volume  = {53},
  pages   = {234--243},
  year    = {2017},
  doi     = {10.1007/s10559-017-9923-9}
}

\end{document}